\documentclass[10pt,oneside,reqno,twoside]{amsart}
\usepackage[foot]{amsaddr} 
\usepackage[utf8]{inputenc}
\usepackage{amsmath}
\usepackage{amsfonts}
\usepackage{amssymb}
\usepackage{setspace}
\usepackage{amssymb}
\usepackage[shortlabels]{enumitem}
\usepackage{xcolor}
\usepackage{graphicx}
\usepackage{appendix}
\usepackage{amsthm}
\newtheorem{thm}{Theorem}[section]
\newtheorem{df}{Definition}[section]

\numberwithin{equation}{section}
\usepackage{bbm}
\newtheorem{rmk}{Remark}[section]

\newtheorem{prop}{Proposition}[section]
\newtheorem{lm}{Lemma}[section]

\usepackage{geometry}
\usepackage{fancyhdr}
\usepackage[backref]{hyperref}
\hypersetup{
	colorlinks=true,
	linkcolor=blue,
	filecolor=magenta,      
	urlcolor=blue,
	pdftitle={Overleaf Example},
	citecolor=red,
	pdfpagemode=FullScreen,
}
\usepackage{mfirstuc}
\makeatletter
\def\@setauthors{%
	\begingroup
	\def\thanks{\protect\thanks@warning}%
	\trivlist
	\centering\footnotesize \@topsep30\p@\relax
	\advance\@topsep by -\baselineskip
	\item\relax
	\author@andify\authors
	\def\\{\protect\linebreak}%
	\authors%
	\ifx\@empty\contribs
	\else
	,\penalty-3 \space \@setcontribs
	\@closetoccontribs
	\fi
	\endtrivlist
	\endgroup
}
\date{}
\begin{document}
	
	\title[Insensitizing controls for stochastic parabolic equations]{Insensitizing controls for stochastic parabolic equations with dynamic boundary conditions}
	
	\author[M. Baroun]{\large Mahmoud Baroun$^1$}
	\address{$^1$Cadi Ayyad University, Faculty of Sciences Semlalia, LMDP, Marrakesh, Morocco.}
	\email{m.baroun@uca.ac.ma}
	\email{abdoelgrou@gmail.com}
	
	\author[S. Boulite]{\large Said Boulite$^2$}
	\address{$^2$Cadi Ayyad University, National School of Applied Sciences, LMDP, Marrakesh, Morocco.}
	\email{s.boulite@uca.ma}
	
	\author[A. Elgrou]{\large Abdellatif Elgrou$^1$}

	\author[O. Oukdach]{\large Omar Oukdach$^3$}
	\address{$^3$Moulay Ismaïl University of Meknes, FST Errachidia,  PMA Laboratory, BP 50, Boutalamine, Errachidia, Morocco.}
	\email{omar.oukdach@gmail.com} 
	
	\dedicatory{ \large Dedicated to the memory of Professor Hammadi Bouslous}
	
	\begin{abstract}
		In this paper, we establish the existence of insensitizing controls for forward linear stochastic parabolic equations with dynamic boundary conditions. We begin by reducing the insensitizing control problem to a classical controllability property of a cascade system of coupled forward-backward stochastic parabolic equations. Next, we derive a suitable observability inequality for the adjoint system using a new Carleman estimate. Finally, employing the classical duality arguments, we solve the considered control problem.
	\end{abstract}
	\maketitle

	\section{Introduction and Main Results}
	Among the well-known control problems in the literature, insensitizing control is one of the most interesting. In this type of problem, the goal is to find a control (or multiple controls) that makes an energy functional locally invariant with respect to small perturbations or measurement errors, such as those in the given initial or boundary values. In this paper, we are interested in analyzing an insensitizing control problem with respect to perturbations of initial values for stochastic parabolic equations with dynamic boundary conditions.
	
	To formulate the problem under consideration, we need to introduce  some necessary notations. Let $T>0$, $G\subset\mathbb{R}^N$ be a given nonempty bounded domain with a smooth boundary $\Gamma=\partial G$, $N\geq2$. Take $G_0$ and $\mathcal{O}$ are some given nonempty open subset of $G$ such that 
	$G_0\cap\mathcal{O}\neq\emptyset$, and $\mathcal{O}_\Gamma$ is a nonempty open subset of $\Gamma$. We indicate by $\chi_{A}$ the characteristic function of a set $A$. Put 
	$$Q=(0,T)\times G, \,\quad \Sigma=(0,T)\times\Gamma,\,\quad \textnormal{and}\,\quad Q_0=(0,T)\times G_0.$$
	
	Let $(\Omega,\mathcal{F},\{\mathcal{F}_t\}_{t\geq0},\mathbb{P})$ be a fixed complete filtered probability space on which a one-dimensional standard Brownian motion $W(\cdot)$ is defined such that $\{\mathcal{F}_t\}_{t\geq0}$ is the natural filtration generated by $W(\cdot)$ and augmented by all the $\mathbb{P}$-null sets in $\mathcal{F}$. Let $H$ be a Banach space, and let $C([0,T];H)$ be the Banach space of all $H$-valued continuous functions defined on $[0,T]$; and for some sub-sigma algebra $\mathcal{G}\subset\mathcal{F}$, we denote by $L^2_{\mathcal{G}}(\Omega;H)$ the Banach space of all $H$-valued $\mathcal{G}$-measurable random variables $X$ such that $\mathbb{E}\vert X\vert_H^2<\infty$, with the canonical norm; and by $L^2_\mathcal{F}(0,T;H)$ the Banach space consisting of all $H$-valued $\{\mathcal{F}_t\}_{t\geq0}$-adapted processes $X(\cdot)$ such that $\mathbb{E}(\vert X(\cdot)\vert^2_{L^2(0,T;H)})<\infty$, with the canonical norm; and by $L^\infty_\mathcal{F}(0,T;H)$ the Banach space consisting of all $H$-valued $\{\mathcal{F}_t\}_{t\geq0}$-adapted essentially bounded processes; and by $L^2_\mathcal{F}(\Omega;C([0,T];H))$ the Banach space consisting of all $H$-valued $\{\mathcal{F}_t\}_{t\geq0}$-adapted continuous processes $X(\cdot)$ such that $\mathbb{E}(\vert X\vert^2_{C([0,T];H)})<\infty$, with the canonical norm.  By considering the Lebesgue measure \( dx \) in \( G \) and the surface measure \( d\sigma \) on \( \Gamma \), we define the bulk-surface space \( \mathbb{L}^2 \) as 
	\[
	\mathbb{L}^2 := L^2(G, dx) \times L^2(\Gamma, d\sigma).
	\]
	Equipped with the inner product
	\[
	\langle (y, y_\Gamma), (z, z_\Gamma) \rangle_{\mathbb{L}^2} = (y, z)_{L^2(G)} + (y_\Gamma, z_\Gamma)_{L^2(\Gamma)},
	\]
	\( \mathbb{L}^2 \) is a Hilbert space.\\
	
	We consider the following forward stochastic parabolic equation subject to dynamic boundary conditions:
	\begin{equation}\label{ass15}
		\begin{cases}
			\begin{array}{ll}
				dy - \triangle y \,dt = (\xi_1+a_1 y+\chi_{G_0}u) \,dt + (a_2y+v_1) \,dW(t) &\textnormal{in}\,\,Q,\\
				dy_\Gamma-\triangle_\Gamma y_\Gamma \,dt+\partial_\nu y \,dt = (\xi_2+b_1y_\Gamma)\,dt+(b_2y_\Gamma+v_2) \,dW(t) &\textnormal{on}\,\,\Sigma,\\
				y_\Gamma(t,x)=y\vert_\Gamma(t,x) &\textnormal{on}\,\,\Sigma,\\
				(y,y_\Gamma)\vert_{t=0}=(y_0+\tau_1\widehat{y}_0,y_{0,\Gamma}+\tau_2\widehat{y}_{0,\Gamma})&\textnormal{in}\,\,G\times\Gamma.
			\end{array}
		\end{cases}
	\end{equation} 
	In equation \eqref{ass15}, $y\vert_\Gamma$ denotes the trace of the function $y$, and $\partial_\nu y = (\nabla y \cdot \nu)\vert_{\Sigma}$ is the normal derivative of $y$, where $\nu$ is the outer unit normal vector at the boundary $\Gamma.$ $\triangle_\Gamma$ is the Beltrami laplacian. The  potentials $a_1$, $a_2$, $b_1$ and $b_2$ are assumed to be bounded i.e.,
	$$a_1, a_2\in L_\mathcal{F}^\infty(0,T;L^\infty(G)), \qquad b_1, b_2\in L_\mathcal{F}^\infty(0,T;L^\infty(\Gamma)).$$ 
	The random variable $(y_0,y_{0,\Gamma})\in L^2_{\mathcal{F}_0}(\Omega;\mathbb{L}^2)$ is a given initial data; $\xi_1\in L^2_\mathcal{F}(0,T;L^2(G))$, $\xi_2\in L^2_\mathcal{F}(0,T;L^2(\Gamma))$ are the known heat source terms. We assume that $(\widehat{y}_0,\widehat{y}_{0,\Gamma})\in L^2_{\mathcal{F}_0}(\Omega;\mathbb{L}^2)$ is unknown with $|(\widehat{y}_0,\widehat{y}_{0,\Gamma})|_{L^2_{\mathcal{F}_0}(\Omega;\mathbb{L}^2)}=1$ and  $\tau_1,\tau_2\in\mathbb{R}$ are unknown small parameters. In this way, $\tau_1 \widehat{y}_0$ and $\tau_1 \widehat{y}_{0,\Gamma}$ can be regarded as the small perturbations of the initial data $y_0$ and $y_{0,\Gamma}$, respectively.  The control function consists of the triple
	$$(u,v_1,v_2) \in L^2_\mathcal{F}(0,T;L^2(G_0))\times L^2_\mathcal{F}(0,T;L^2(G))\times L^2_\mathcal{F}(0,T;L^2(\Gamma)).$$
	In contrast to the deterministic case, the extra controls $v_1$ and $v_2$ are needed. This is due to some well-known technical issues when dealing with the controllability of forward stochastic parabolic equations. We refer to Remark \ref{rmqq1.33} for a discussion of the possibility to consider only one control.\\
	
	Define the following energy functional for equation \eqref{ass15}:
	\begin{align}\label{functioPhii1.2}
		\Phi(y,y_\Gamma)=\frac{1}{2}\mathbb{E}\int_0^T\int_{\mathcal{O}} |y|^2 \,dx\,dt+\frac{1}{2}\mathbb{E}\int_0^T\int_{\mathcal{O}_\Gamma} |y_\Gamma|^2 \,d\sigma\,dt,
	\end{align}
	where $(y,y_\Gamma)=(y(t,x;\tau_1,\tau_2,u,v_1,v_2),y_\Gamma(t,x;\tau_1,\tau_2,u,v_1,v_2))$ is the solution of \eqref{ass15} associated to the parameters $\tau_1$, $\tau_2$, and the controls $u, v_1$ and $v_2$.
	The problem of
	insensitizing means that we are seeking controls $(u, v_1, v_2)$ such that
	the energy  $\Phi$ is locally invariant for small perturbations in the initial data. More precisely, the  insensitizing control problem for  \eqref{ass15} is defined as follows.
	\begin{df}\label{deff1}
		For any $\xi_1\in L^2_\mathcal{F}(0,T;L^2(G))$, $\xi_2\in L^2_\mathcal{F}(0,T;L^2(\Gamma))$, and $(y_0,y_{0,\Gamma})\in L^2_{\mathcal{F}_0}(\Omega;\mathbb{L}^2)$, a control triple $(u,v_1,v_2)$ is said to insensitize the functional $\Phi$ if
		\begin{align}\label{inspb}
			\frac{\partial \Phi(y,y_\Gamma)}{\partial\tau_1}\Bigg|_{\tau_1=\tau_2=0} = 0 \quad \text{and} \quad 
			\frac{\partial \Phi(y,y_\Gamma)}{\partial\tau_2}\Bigg|_{\tau_1=\tau_2=0} = 0,
		\end{align}
		for all $(\widehat{y}_0,\widehat{y}_{0,\Gamma})\in L^2_{\mathcal{F}_0}(\Omega;\mathbb{L}^2)$ satisfying 
		\[
		|(\widehat{y}_0,\widehat{y}_{0,\Gamma})|_{L^2_{\mathcal{F}_0}(\Omega;\mathbb{L}^2)} = 1.
		\]
	\end{df}
	
	The main result of this paper is the existence of insensitizing controls for \eqref{ass15}, stated as follows.
	\begin{thm}\label{thmm1.3ins}
		Assume that $G_0\cap\mathcal{O}\neq\emptyset$ and $(y_0,y_{0,\Gamma})=(0,0)$. Then, there exist constants $M > 0$ and $C > 0$ depending only on $G$, $G_0$, $\mathcal{O}$, $T$, $a_1$, $a_2$, $b_1$, and $b_2$ such that for any 
		\[
		\xi_1\in L^2_\mathcal{F}(0,T;L^2(G)), \quad \xi_2\in L^2_\mathcal{F}(0,T;L^2(\Gamma)),
		\]
		satisfying 
		\begin{align}
			\left|\mathrm{exp}(Mt^{-1})\xi_1\right|_{L^2_\mathcal{F}(0,T;L^2(G))} + \left|\mathrm{exp}(Mt^{-1})\xi_2\right|_{L^2_\mathcal{F}(0,T;L^2(\Gamma))} < \infty,
		\end{align}
		there exists a control triple 
		\[
		(u,v_1,v_2) \in L^2_\mathcal{F}(0,T;L^2(G_0))\times L^2_\mathcal{F}(0,T;L^2(G))\times L^2_\mathcal{F}(0,T;L^2(\Gamma)),
		\]
		which insensitizes the functional $\Phi$ in the sense of Definition \ref{deff1}. Moreover
		\begin{align}\label{estiof controls}
			\begin{aligned}
				&|u|_{L^2_\mathcal{F}(0,T;L^2(G_0))} + |v_1|_{L^2_\mathcal{F}(0,T;L^2(G))} + |v_2|_{L^2_\mathcal{F}(0,T;L^2(\Gamma))} \\
				&\leq C\left(\left|\mathrm{exp}(Mt^{-1})\xi_1\right|_{L^2_\mathcal{F}(0,T;L^2(G))} + \left|\mathrm{exp}(Mt^{-1})\xi_2\right|_{L^2_\mathcal{F}(0,T;L^2(\Gamma))}\right).
			\end{aligned}
		\end{align}
	\end{thm}
	\vspace{0.3cm}
	Insensitivity properties for dynamic systems are the subject of extensive research in the literature. In the context of deterministic parabolic equations with static boundary conditions, this subject has been extensively studied. To our knowledge, the seminal work by J.-L. Lions in \cite{lions1989quel} was the first to address this question. Since then, many studies have followed, including \cite{bodafabre95,BodarBurgosperez,BodarBurgosPerez2004,BodarBurgosPerez04NonAnalysis,BodaGnPer,Tereza2000,Tereza97Esaim}, among others. In particular, \cite{Tereza2000} demonstrated that the existence of insensitizing controls is not guaranteed for all initial data, while \cite{Tere_idenfication} investigated the class of initial data that can be insensitized. Additionally, \cite{gureSiam07} considered a functional involving the gradient of the state for the linear heat equation. For parabolic equations with dynamic boundary conditions, the insensitizing control problems are not sufficiently addressed. As far as we know, the only existing papers in this context are \cite{sancarrmora,zhaYinGao}.
	
	We point out that the previously mentioned articles focus on deterministic equations. In the stochastic setting, to the best of our knowledge, \cite{liu2014global,liu2019carleman,yansun2011} are the only papers addressing insensitivity properties for the stochastic heat equation with Dirichlet boundary conditions. Accordingly, the present paper aims to investigate, for the first time, the question of insensitizing controls for stochastic parabolic equations with dynamic boundary conditions.
	
	Recently, stochastic parabolic equations with dynamic boundary conditions have gained attention in various contexts. For some results on controllability problems, see \cite{elgrou22SPEwithDBC,BackSPEwithDBC,elgrou1D23}. Multi-objective control problems have been discussed in \cite{BEMOstoch}, while inverse problems are treated in \cite{InvStoch}. For results on inverse problems in various contexts, see for instance \cite{ACMO20,kiran,Yamam2009invePrb} and references therein. Additionally, some related works in the deterministic setting include \cite{BoMaOuCost,BoMaOuNash,khoutaibi2020null1,maniar2017null,Spectral,BoMaOuNash2}.\\
	
	We now provide some remarks.
	\begin{rmk}
		Notice that the insensitivity control problem \eqref{inspb} has been established in Theorem \ref{thmm1.3ins} under the initial data $(y_0,y_{0,\Gamma})=(0,0)$. In \cite{Tereza2000}, the author proved that when $G_0\Subset G$, there exists an initial data $(y_0,y_{0,\Gamma})\in L^2_{\mathcal{F}_0}(\Omega;\mathbb{L}^2)$ such that the functional $\Phi$ given in \eqref{functioPhii1.2} cannot be insensitized. Therefore, in order to provide a positive answer to \eqref{inspb} for every $(y_0,y_{0,\Gamma})\in L^2_{\mathcal{F}_0}(\Omega;\mathbb{L}^2)$, we can modify the functional $\Phi$ as follows: For any $t_0>0$, we consider 
		\begin{align*}
			\Phi_{t_0}(y,y_\Gamma)=\frac{1}{2}\mathbb{E}\int_{t_0}^T\int_{\mathcal{O}} |y|^2 \,dx\,dt+\frac{1}{2}\mathbb{E}\int_{t_0}^T\int_{\mathcal{O}_\Gamma} |y_\Gamma|^2 \,d\sigma\,dt.
		\end{align*}
		Using some ideas from the proof of \cite[Corollary 1]{Tereza2000} and Theorem \ref{thmm1.3ins}, we have the following result: Assume $G_0\cap\mathcal{O}\neq\emptyset$, and let $\xi_1\in L^2_\mathcal{F}(0,T;L^2(G))$ and $\xi_2\in L^2_\mathcal{F}(0,T;L^2(\Gamma))$ be such that $\xi_1(t,x)=\xi_2(t,x)=0$ for $t\in(0,t_0)$. Then, for every $(y_0,y_{0,\Gamma})\in L^2_{\mathcal{F}_0}(\Omega;\mathbb{L}^2)$, there exists a control triple $(u,v_1,v_2) \in L^2_\mathcal{F}(0,T;L^2(G_0))\times L^2_\mathcal{F}(0,T;L^2(G))\times L^2_\mathcal{F}(0,T;L^2(\Gamma))$
		that insensitizes the functional $\Phi_{t_0}$.
	\end{rmk}
	
	\begin{rmk}
		The geometric condition $G_0\cap\mathcal{O}\neq\emptyset$ is imposed as a sufficient condition in Theorem \ref{thmm1.3ins} to establish the insensitivity control property \eqref{inspb}. However, it is also of significant interest to study the case when $G_0\cap\mathcal{O}=\emptyset$. For further discussion and remarks in this direction, see \cite{ErveLissPriva}.
	\end{rmk}
	
	\begin{rmk}\label{rmqq1.33}
		In equation \eqref{ass15}, we have acted with one localized control $u$ and two extra controls $v_1$ and $v_2$, which are needed due to the classical controllability problem for forward stochastic parabolic equations (see \cite{elgrou22SPEwithDBC,lu2021mathematical,tang2009null}), and equivalently, to the observability problem of the (adjoint) backward equations. Thus, it is  important to study the case when we act only with the control $u$ on equation \eqref{ass15} to insensitize the functional $\Phi$. In this state of mind, it is meaningful to try the spectral approach as developed in \cite{withouextra,lu2011some,angZhong16}. 
	\end{rmk}

	This paper is organized as follows. In Section \ref{sec03}, we present the well-posedness results for our equations. In Section \ref{sec3}, we reduce the insensitizing control problem to a classical controllability problem for a coupled system of forward-backward stochastic parabolic equations. In Section \ref{sec4}, we establish a global Carleman estimate for the adjoint coupled system. Section \ref{sec5} is devoted to proving our main result, i.e., Theorem \ref{thmm1.3ins}.
	\section{Functional setting and well-posedness}\label{sec03}In this section, we provide the well-posedness results of the forward and backward stochastic parabolic equations with dynamic boundary conditions. Let us first define some useful differential operators on $\Gamma$. Firstly, the tangential gradient of a function \(y_\Gamma: \Gamma \rightarrow \mathbb{R}\) is defined by
	\[
	\nabla_{\Gamma} y_\Gamma = \nabla y - \partial_{\nu} y \, \nu,
	\]
	where \(y\) is an extension of \(y_\Gamma\) up to an open neighborhood of \(\Gamma\). Secondly, the tangential divergence of a function \(y_\Gamma: \Gamma \rightarrow \mathbb{R}^N\), such that \(y = y_\Gamma\) on \(\Gamma\), is defined as
	\[
	\operatorname{div}_{\Gamma} (y_\Gamma) = \operatorname{div}(y) - \nabla y \,\nu  \cdot \nu,
	\]
	where \(\nabla y\) is the standard gradient of \(y\). Then, we define the Laplace-Beltrami operator of a function \(y_\Gamma: \Gamma \rightarrow \mathbb{R}\) as follows:
	\[
	\Delta_{\Gamma} y_\Gamma = \operatorname{div}_{\Gamma} (\nabla_{\Gamma} y_\Gamma).
	\]
	As for the standard Sobolev spaces $H^1(G)$ and $H^2(G)$, we  define the following Sobolev spaces $H^1(\Gamma)$ and $H^2(\Gamma)$ on $\Gamma$ by 
	$$H^1(\Gamma)=\Big\{y_\Gamma\in L^2(\Gamma),\;\;\nabla_\Gamma y_\Gamma\in L^2(\Gamma;\mathbb{R}^N)\Big\},\quad H^2(\Gamma)=\Big\{y_\Gamma\in L^2(\Gamma),\;\;\nabla_\Gamma y_\Gamma\in H^1(\Gamma;\mathbb{R}^N)\Big\},$$
	endowed with the following norms respectively
	$$ |y_\Gamma|_{H^1(\Gamma)}= \sqrt{\langle y_\Gamma,y_\Gamma \rangle_{H^1(\Gamma)}}, \quad\text{with} \quad\,\,\langle y_\Gamma,z_\Gamma \rangle_{H^1(\Gamma)}= \int_{\Gamma}y_\Gamma z_\Gamma \,d\sigma + \int_{\Gamma}\nabla_{\Gamma} y_\Gamma\cdot \nabla_{\Gamma}z_\Gamma \,d\sigma, \quad $$
	and 
	$$ |y_\Gamma|_{H^2(\Gamma)}= \sqrt{\langle y_\Gamma,y_\Gamma \rangle_{H^2(\Gamma)}},\quad \text{with}\quad\,\, \langle y_\Gamma,z_\Gamma \rangle_{H^2(\Gamma)}= \int_{\Gamma}y_\Gamma z_\Gamma \,d\sigma +\int_{\Gamma}\Delta_{\Gamma} y_\Gamma\, \Delta_{\Gamma}z_\Gamma \,d\sigma.$$              
	In the following sections, for the Laplace-Belrami operator $\Delta_\Gamma$, we frequently use the following surface divergence formula
	$$
	\int_\Gamma \triangle_\Gamma y_\Gamma\, z_\Gamma\, \,d\sigma = -\int_\Gamma \nabla_\Gamma y_\Gamma\cdot\nabla_\Gamma z_\Gamma\, \,d\sigma\,,\qquad y_\Gamma\in H^2(\Gamma),\;\; z_\Gamma\in H^1(\Gamma).
	$$   
	We also introduce the following Hilbert spaces
	$$\mathbb{H}^k=\Big\{(y,y_\Gamma)\in H^k(G)\times H^k(\Gamma): \,\,\; y_\Gamma=y|_\Gamma\Big\},\quad k=1,2,$$
	viewed as a subspace of $H^k(G)\times H^k(\Gamma)$ with the natural topology inherited from $H^k(G)\times H^k(\Gamma)$.
	
	Let us first consider the following  forward stochastic parabolic equation
	\begin{equation}\label{eqqwf1}
		\begin{cases}
			\begin{array}{ll}
				dy - \Delta y \,dt = (a_1y+f_1) \,dt + (a_2y +f_2)\,dW(t)&\textnormal{in}\,\,Q,\\
				dy_\Gamma-\Delta_\Gamma y_\Gamma \,dt+\partial_\nu y \,dt = (b_1y_\Gamma+g_1) \,dt+(b_2y_\Gamma + g_2) \,dW(t) &\textnormal{on}\,\,\Sigma,\\
				y_\Gamma(t,x)=y\vert_\Gamma(t,x) &\textnormal{on}\,\,\Sigma,\\
				(y,y_\Gamma)\vert_{t=0}=(y_0,y_{0,\Gamma}) &\textnormal{in}\,\,G\times\Gamma,
			\end{array}
		\end{cases}
	\end{equation}
	where $(y_0,y_{0,\Gamma})\in L^2_{\mathcal{F}_0}(\Omega;\mathbb{L}^2)$ is the initial state, $(y,y_\Gamma)$ is the state variable, $a_1, a_2\in L_\mathcal{F}^\infty(0,T;L^\infty(G))$, $b_1, b_2\in L_\mathcal{F}^\infty(0,T;L^\infty(\Gamma))$, $f_1,f_2\in L^2_\mathcal{F}(0,T;L^2(G))$ and $g_1,g_2\in L^2_\mathcal{F}(0,T;L^2(\Gamma))$.
	
	From \cite[Theorem 2.1]{elgrou22SPEwithDBC}, we have the following well-posedness of \eqref{eqqwf1}.
	\begin{thm}\label{wellposedness1}
		For each  $(y_0,y_{0,\Gamma})\in L^2_{\mathcal{F}_0}(\Omega;\mathbb{L}^2)$, $f_1,f_2\in L^2_{\mathcal{F}}(0,T;L^2(G))$, and $g_1,g_2\in L^2_{\mathcal{F}}(0,T;L^2(\Gamma))$, the equation \eqref{eqqwf1} admits a unique weak solution
		$$(y,y_\Gamma)\in L^2_\mathcal{F}(\Omega;C([0,T];\mathbb{L}^2))\bigcap L^2_\mathcal{F}(0,T;\mathbb{H}^1).$$
		Moreover, there exists a constant $C>0$ depending only on $G$, $T$, $a_1, \,a_2, \,b_1$, and $b_2$ so that
		\begin{align*}
			\begin{aligned}
				&\,\vert(y,y_\Gamma)\vert_{L^2_\mathcal{F}(\Omega;C([0,T];\mathbb{L}^2))} + \vert(y,y_\Gamma)\vert_{L^2_\mathcal{F}(0,T;\mathbb{H}^1)}\\
				&\leq C\,\Big(|(y_0,y_{0,\Gamma})|_{L^2_{\mathcal{F}_0}(\Omega;\mathbb{L}^2)}+|f_1|_{L^2_{\mathcal{F}}(0,T;L^2(G))}+|f_2|_{L^2_{\mathcal{F}}(0,T;L^2(G))}\\
				&\hspace{0.9cm}\,+|g_1|_{L^2_{\mathcal{F}}(0,T;L^2(\Gamma))}+|g_2|_{L^2_{\mathcal{F}}(0,T;L^2(\Gamma))}\Big).
			\end{aligned}
		\end{align*}
	\end{thm}
	On the other hand, we introduce the following backward stochastic parabolic equation
	\begin{equation}\label{eqqwb1}
		\begin{cases}
			\begin{array}{ll}
				dz + \Delta z \,dt = (a_3z+a_4Z+f_3) \,dt + Z\,dW(t)&\textnormal{in}\,\,Q,\\
				dz_\Gamma+\Delta_\Gamma z_\Gamma \,dt-\partial_\nu z \,dt = (b_3z_\Gamma+b_4\widehat{Z}+g_3) \,dt+\widehat{Z} \,dW(t) &\textnormal{on}\,\,\Sigma,\\
				z_\Gamma(t,x)=z\vert_\Gamma(t,x) &\textnormal{on}\,\,\Sigma,\\
				(z,z_\Gamma)\vert_{t=T}=(z_T,z_{T,\Gamma}) &\textnormal{in}\,\,G\times\Gamma,
			\end{array}
		\end{cases}
	\end{equation}
	where $(z_T,z_{T,\Gamma})\in L^2_{\mathcal{F}_T}(\Omega;\mathbb{L}^2)$ is the terminal state, $(z,z_\Gamma,Z,\widehat{Z})$ is the state variable, $a_3, a_4\in L_\mathcal{F}^\infty(0,T;L^\infty(G))$, $b_3, b_4\in L_\mathcal{F}^\infty(0,T;L^\infty(\Gamma))$, $f_3\in L^2_\mathcal{F}(0,T;L^2(G))$ and $g_3\in L^2_\mathcal{F}(0,T;L^2(\Gamma))$.
	
	From \cite[Theorem 2.2]{elgrou22SPEwithDBC}, we give the following well-posedness result of \eqref{eqqwb1}.
	\begin{thm}\label{wellposedness2}
		For each  $(z_T, z_{T,\Gamma}) \in L_{\mathcal{F}_T}^2\left(\Omega ; \mathbb{L}^2\right)$, $f_3\in L^2_\mathcal{F}(0,T;L^2(G))$ and $g_3\in L^2_\mathcal{F}(0,T;L^2(\Gamma))$, the system \eqref{eqqwb1} has a unique weak solution
		$$
		(z, z_{\Gamma}, Z, \widehat{Z}) \in\left(L_{\mathcal{F}}^2\left(\Omega ; C([0, T]; \mathbb{L}^2)\right) \bigcap L_{\mathcal{F}}^2\left(0, T ; \mathbb{H}^1\right)\right) \times L_{\mathcal{F}}^2\left(0, T ; \mathbb{L}^2\right).
		$$
		Furthermore, there exists a constant $C>0$ depending only on $G$, $T$, $a_3, \,a_4, \,b_3$, and $b_4$  such that
		\begin{align*}
			\begin{aligned}
				&\,\left|\left(z, z_{\Gamma}\right)\right|_{L_{\mathcal{F}}^2\left(\Omega ; C\left([0, T] ; \mathbb{L}^2\right)\right)}+\left|\left(z, z_{\Gamma}\right)\right|_{L_{\mathcal{F}}^2\left(0, T ; \mathbb{H}^1\right)}+|(Z, \widehat{Z})|_{L_{\mathcal{F}}^2\left(0, T ; \mathbb{L}^2\right)}\\
				&\leq C\Big(\left|\left(z_T, z_{T,\Gamma}\right)\right|_{L_{\mathcal{F}_T}^2\left(\Omega ; \mathbb{L}^2\right)}+|f_3|_{L^2_\mathcal{F}(0,T;L^2(G))} +|g_3|_{L^2_\mathcal{F}(0,T;L^2(\Gamma))}\Big).
			\end{aligned}
		\end{align*}
	\end{thm}
	
	\section{Reduction of the insensitizing control problem}\label{sec3}
	This section is devoted to reducing the insensitivity control problem \eqref{inspb} to the following partial null controllability problem for a cascade system of coupled forward-backward stochastic parabolic equations.
	\begin{prop}\label{proposs11}
		For  $\xi_1\in L^2_\mathcal{F}(0,T;L^2(G))$, $\xi_2\in L^2_\mathcal{F}(0,T;L^2(\Gamma))$, and $(y_0,y_{0,\Gamma})\in L^2_{\mathcal{F}_0}(\Omega;\mathbb{L}^2)$, let $(y,y_\Gamma;z,z_\Gamma,Z,\widehat{Z})$ be the solution of the following linear cascade system of coupled stochastic parabolic equations associated with the controls $(u,v_1,v_2) \in L^2_\mathcal{F}(0,T;L^2(G_0))\times L^2_\mathcal{F}(0,T;L^2(G))\times L^2_\mathcal{F}(0,T;L^2(\Gamma))$
		\begin{equation}\label{forr4.1}
			\begin{cases}
				\begin{array}{ll}
					dy - \triangle y \,dt = (\xi_1+a_1 y+\chi_{G_0}u) \,dt + (a_2y+v_1) \,dW(t) &\textnormal{in}\,\,Q,\\
					dy_\Gamma-\triangle_\Gamma y_\Gamma \,dt+\partial_\nu y \,dt = (\xi_2+b_1y_\Gamma)\,dt+(b_2y_\Gamma+v_2) \,dW(t) &\textnormal{on}\,\,\Sigma,\\
					y_\Gamma(t,x)=y\vert_\Gamma(t,x) &\textnormal{on}\,\,\Sigma,\\
					(y,y_\Gamma)\vert_{t=0}=(y_0,y_{0,\Gamma}) &\textnormal{in}\,\,G\times\Gamma,
				\end{array}
			\end{cases}
		\end{equation}
		and 
		\begin{equation}\label{back45}
			\begin{cases}
				\begin{array}{ll}
					dz + \triangle z \,dt = (-a_1 z-a_2Z-\chi_{\mathcal{O}}y) \,dt + Z \,dW(t) &\textnormal{in}\,\,Q,\\
					dz_\Gamma+\triangle_\Gamma z_\Gamma \,dt-\partial_\nu z \,dt = (-b_1z_\Gamma-b_2\widehat{Z}-\chi_{\mathcal{O}_\Gamma}y_\Gamma)\,dt+\widehat{Z} \,dW(t) &\textnormal{on}\,\,\Sigma,\\
					z_\Gamma(t,x)=z\vert_\Gamma(t,x) &\textnormal{on}\,\,\Sigma,\\
					(z,z_\Gamma)\vert_{t=T}=(0,0) &\textnormal{in}\,\,G\times\Gamma.
				\end{array}
			\end{cases}
		\end{equation}
		Then  the insensitivity control problem \eqref{inspb} holds for the control triple $(u,v_1,v_2)$
		if and only if
		the solution $(y,y_\Gamma;z,z_\Gamma,Z,\widehat{Z})$ of \eqref{forr4.1}-\eqref{back45} satisfies that
		\begin{align}\label{null controprop}
			z(0,\cdot)=0\;\; \textnormal{in}\;\; G, \quad z_\Gamma(0,\cdot)=0 \;\; \textnormal{on}\;\;  \Gamma, \quad\mathbb{P}\textnormal{-a.s.}
		\end{align}
	\end{prop}
	\begin{proof}
		For any $(u,v_1,v_2) \in L^2_\mathcal{F}(0,T;L^2(G_0))\times L^2_\mathcal{F}(0,T;L^2(G))\times L^2_\mathcal{F}(0,T;L^2(\Gamma))$, $\tau_1,\tau_2\in\mathbb{R}$ and $(\widehat{y}_0,\widehat{y}_{0,\Gamma})\in L^2_{\mathcal{F}_0}(\Omega;\mathbb{L}^2)$ such that $|(\widehat{y}_0,\widehat{y}_{0,\Gamma})|_{L^2_{\mathcal{F}_0}(\Omega;\mathbb{L}^2)}=1$, we denote by  $(y_\tau,y_{\tau,\Gamma})$ is the solution of \eqref{ass15} associated to $\tau_1,\tau_2$ and the controls $u, v_1$ and $v_2$. It is not difficult to see that
		\begin{align}\label{partdert43}
			\begin{aligned}
				\displaystyle\frac{\partial \Phi(y_\tau,y_{\tau,\Gamma})}{\partial\tau_1}\Bigg|_{\tau_1=\tau_2=0}=\lim_{\tau_1\rightarrow0}\frac{1}{2}&\Bigg[\mathbb{E}\int_0^T\int_\mathcal{O} (y_{\tau_1}+y)\frac{y_{\tau_1}-y}{\tau_1} \,dx\,dt\\
				&\;\;+\mathbb{E}\int_0^T\int_{\mathcal{O}_\Gamma} (y_{\tau_1,\Gamma}+y_\Gamma)\frac{y_{\tau_1,\Gamma}-y_\Gamma}{\tau_1} \,d\sigma\,dt\Bigg],
			\end{aligned}
		\end{align}
		where $(y,y_\Gamma)$ is the solution of \eqref{forr4.1}, and $(y_{\tau_1},y_{\tau_1,\Gamma})$ is the solution of the following equation
		\begin{equation}\label{}
			\begin{cases}
				\begin{array}{ll}
					dy_{\tau_1} - \triangle y_{\tau_1} \,dt = (\xi_1+a_1 y_{\tau_1}+\chi_{G_0}u) \,dt + (a_2y_{\tau_1}+v_1) \,dW(t) &\textnormal{in}\,\,Q,\\
					dy_{\tau_1,\Gamma}-\triangle_\Gamma y_{\tau_1,\Gamma} \,dt+\partial_\nu y_{\tau_1} \,dt = (\xi_2+b_1y_{\tau_1,\Gamma})\,dt+(b_2y_{\tau_1,\Gamma}+v_2) \,dW(t) &\textnormal{on}\,\,\Sigma,\\
					y_{\tau_1,\Gamma}(t,x)=y_{\tau_1}\vert_\Gamma(t,x) &\textnormal{on}\,\,\Sigma,\\
					(y_{\tau_1},y_{\tau_1,\Gamma})\vert_{t=0}=(y_0+\tau_1\widehat{y}_0,y_{0,\Gamma})&\textnormal{in}\,\,G\times\Gamma.
				\end{array}
			\end{cases}
		\end{equation}
		Set $\overline{y}=\frac{y_{\tau_1}-y}{\tau_1}$ and $\overline{y}_\Gamma=\frac{y_{\tau_1,\Gamma}-y_\Gamma}{\tau_1}$, then it is easy to see that $(\overline{y},\overline{y}_\Gamma)$
		is the solution of 
		\begin{equation}\label{forr4.4}
			\begin{cases}
				\begin{array}{ll}
					d\overline{y} - \triangle \overline{y} \,dt = a_1 \overline{y} \,dt + a_2\overline{y} \,dW(t) &\textnormal{in}\,\,Q,\\
					d\overline{y}_\Gamma-\triangle_\Gamma \overline{y}_\Gamma \,dt+\partial_\nu \overline{y} \,dt = b_1\overline{y}_\Gamma\,dt+b_2\overline{y}_\Gamma \,dW(t) &\textnormal{on}\,\,\Sigma,\\
					\overline{y}_\Gamma(t,x)=\overline{y}\vert_\Gamma(t,x) &\textnormal{on}\,\,\Sigma,\\
					(\overline{y},\overline{y}_\Gamma)\vert_{t=0}=(\widehat{y}_0,0) &\textnormal{in}\,\,G\times\Gamma.
				\end{array}
			\end{cases}
		\end{equation}
		From \eqref{partdert43}, notice that
		\begin{align}\label{equdepar36}
			\frac{\partial \Phi(y_\tau,y_{\tau,\Gamma})}{\partial\tau_1}\Bigg|_{\tau_1=\tau_2=0}=\mathbb{E}\int_0^T\int_\mathcal{O} y\overline{y} \,dx\,dt+\mathbb{E}\int_0^T\int_{\mathcal{O}_\Gamma} y_\Gamma\overline{y}_\Gamma \,d\sigma\,dt.
		\end{align}
		Using Itô's formula for solutions of  \eqref{back45} and \eqref{forr4.4}, we obtain that
		\begin{align*}
			\mathbb{E}\int_0^T\int_\mathcal{O} y\overline{y} \,dx\,dt+\mathbb{E}\int_0^T\int_{\mathcal{O}_\Gamma} y_\Gamma\overline{y}_\Gamma \,d\sigma\,dt=\mathbb{E}\int_G \widehat{y}_0 z(0)\,dx,
		\end{align*}
		which, together with \eqref{equdepar36}, implies that
		\begin{align}\label{partdert432.8}
			\frac{\partial \Phi(y_\tau,y_{\tau,\Gamma})}{\partial\tau_1}\Bigg|_{\tau_1=\tau_2=0}=\mathbb{E}\int_G \widehat{y}_0 z(0)\,dx.
		\end{align}
		On the other hand, similarly to $\frac{\partial \Phi}{\partial\tau_1}$, we can also derive that
		\begin{align}\label{partdert43tau2}
			\frac{\partial \Phi(y_\tau,y_{\tau,\Gamma})}{\partial\tau_2}\Bigg|_{\tau_1=\tau_2=0}=\mathbb{E}\int_\Gamma \widehat{y}_{0,\Gamma}z_\Gamma(0) \,d\sigma.
		\end{align}
		Combining \eqref{partdert432.8} and \eqref{partdert43tau2}, the insensitivity problem \eqref{inspb} holds if and only if 
		$$\mathbb{E}\int_G \widehat{y}_0 z(0)\,dx=\mathbb{E}\int_\Gamma \widehat{y}_{0,\Gamma}z_\Gamma(0) \,d\sigma=0,$$
		for all $(\widehat{y}_0,\widehat{y}_{0,\Gamma})\in L^2_{\mathcal{F}_0}(\Omega;\mathbb{L}^2)$ such that $|(\widehat{y}_0,\widehat{y}_{0,\Gamma})|_{L^2_{\mathcal{F}_0}(\Omega;\mathbb{L}^2)}=1$.
		Thus, we conclude that the insensitivity problem \eqref{inspb} is satisfied if and only if the controllability property \eqref{null controprop} holds. This concludes the proof of Proposition \ref{proposs11}.
	\end{proof}
	By the classical duality argument, the null controllability \eqref{null controprop} of equations \eqref{forr4.1}-\eqref{back45} can be reduced to an observability inequality for the following coupled stochastic parabolic equations
	\begin{equation}\label{adback4.77}
		\begin{cases}
			\begin{array}{ll}
				dp + \triangle p \,dt = (-a_1 p-a_2P-\chi_{\mathcal{O}}q) \,dt + P \,dW(t) &\textnormal{in}\,\,Q,\\
				dp_\Gamma+\triangle_\Gamma p_\Gamma \,dt-\partial_\nu p \,dt = (-b_1p_\Gamma-b_2\widehat{P}-\chi_{\mathcal{O}_\Gamma}q_\Gamma)\,dt+\widehat{P} \,dW(t) &\textnormal{on}\,\,\Sigma,\\
				p_\Gamma(t,x)=p\vert_\Gamma(t,x) &\textnormal{on}\,\,\Sigma,\\
				(p,p_\Gamma)\vert_{t=T}=(0,0) &\textnormal{in}\,\,G\times\Gamma,
			\end{array}
		\end{cases}
	\end{equation}
	and
	\begin{equation}\label{adjoforr4.8}
		\begin{cases}
			\begin{array}{ll}
				dq - \triangle q \,dt = a_1 q \,dt + a_2q \,dW(t) &\textnormal{in}\,\,Q,\\
				dq_\Gamma-\triangle_\Gamma q_\Gamma \,dt+\partial_\nu q \,dt = b_1q_\Gamma\,dt+b_2q_\Gamma\,dW(t) &\textnormal{on}\,\,\Sigma,\\
				q_\Gamma(t,x)=q\vert_\Gamma(t,x) &\textnormal{on}\,\,\Sigma,\\
				(q,q_\Gamma)\vert_{t=0}=(q_0,q_{0,\Gamma}) &\textnormal{in}\,\,G\times\Gamma.
			\end{array}
		\end{cases}
	\end{equation}
	
	\section{Carleman estimates and the needed observability inequality}\label{sec4}
	This section is devoted to establishing a new Carleman estimate for the coupled system of forward-backward stochastic parabolic equations \eqref{adback4.77}-\eqref{adjoforr4.8}. This Carleman estimate will be the key tool in deriving our desired observability inequality.
	
	\subsection{Carleman estimate for the coupled system \eqref{adback4.77}-\eqref{adjoforr4.8}}
	
	We first introduce the following known lemma. For the proof, we refer to \cite{fursikov1996controllability}.
	\begin{lm}\label{lmm5.1}
		For any nonempty open subset $G_1\Subset G$ $($i.e., $\overline{G_1}\subset G$$)$, there exists a function $\psi\in C^4(\overline{G})$ such that
		$$
		\psi>0\;\,\, \textnormal{in} \,\,G\,;\qquad \psi=0\;\,\,\, \textnormal{on} \,\,\Gamma;\qquad\vert\nabla\psi\vert>0\; \,\,\,\,\textnormal{in}\,\,\overline{G\setminus G_1}.
		$$
	\end{lm}
	Since $\vert\nabla\psi\vert^2=\vert\nabla_\Gamma\psi\vert^2+\vert\partial_\nu\psi\vert^2$ on $\Gamma$,  the function $\psi$ in Lemma \ref{lmm5.1} satisfies that
	\begin{equation*}
		\nabla_\Gamma\psi=0\,,\qquad \vert\nabla\psi\vert=\vert\partial_\nu\psi\vert\,,\qquad\partial_\nu\psi\leq -C<0\,\,\,\;\;\textnormal{on}\,\,\;\Gamma, \,\,\,\textnormal{for some constant}\,\,C>0.
	\end{equation*}
	For large parameters $\lambda, \mu\geq1$, we choose the following weight functions
	\begin{align*}
		&\,\alpha\equiv \alpha(t,x) = \frac{e^{\mu\psi(x)}-e^{2\mu\vert\psi\vert_\infty}}{t(T-t)},\qquad \gamma\equiv \gamma(t) = \frac{1}{t(T-t)},\qquad\theta\equiv\theta(t,x)=e^{\lambda\alpha}.
	\end{align*}
	
	On the other hand, it is easy to check that there exists a constant $C=C(G,T)>0$ such that for all $(t,x)\in Q$
	\begin{align*}\begin{aligned}
			&\,\gamma(t)\geq C,\qquad\vert\gamma_t(t)\vert\leq C\gamma^2(t),\qquad\vert\gamma_{tt}(t)\vert\leq C\gamma^3(t),\\
			&\vert\alpha_t(t,x)\vert\leq Ce^{2\mu\vert\psi\vert_\infty}\gamma^2(t),\qquad\vert\alpha_{tt}(t,x)\vert\leq Ce^{2\mu\vert\psi\vert_\infty}\gamma^3(t).
	\end{aligned}\end{align*}
	
	We first consider the following forward stochastic parabolic equation
	\begin{equation}\label{forq5.1}
		\begin{cases}
			\begin{array}{ll}
				dq - \triangle q \,dt = f \,dt + g \,dW(t) &\textnormal{in}\,\,Q,\\
				dq_\Gamma-\triangle_\Gamma q_\Gamma \,dt+\partial_\nu q \,dt = f_\Gamma\,dt+g_\Gamma \,dW(t) &\textnormal{on}\,\,\Sigma,\\
				q_\Gamma(t,x)=q\vert_\Gamma(t,x) &\textnormal{on}\,\,\Sigma,\\
				(q,q_\Gamma)\vert_{t=0}=(q_0,q_{0,\Gamma}) &\textnormal{in}\,\,G\times\Gamma,
			\end{array}
		\end{cases}
	\end{equation}
	where $(q_0,q_{0,\Gamma})\in L^2_{\mathcal{F}_0}(\Omega;\mathbb{L}^2)$ is the initia condition, $f, g\in L^2_\mathcal{F}(0,T;L^2(G))$, $f_\Gamma, g_\Gamma\in L^2_\mathcal{F}(0,T;L^2(\Gamma))$. We have the following known Carleman estimate for \eqref{forq5.1} (see \cite[Theorem 3.1]{elgrou22SPEwithDBC}).
	\begin{lm}\label{lmm4.2}
		There exists a large $\mu_1\geq1$ such that for $\mu=\mu_1$, one can find a constant $C>0$ and a large $\lambda_1\geq1$ depending only on $G$, $G_1$, $\mu_1$ and $T$ such that for all $\lambda\geq\lambda_1$, $f, g\in L^2_\mathcal{F}(0,T;L^2(G))$, $f_\Gamma, g_\Gamma\in L^2_\mathcal{F}(0,T;L^2(\Gamma))$, and $(q_0,q_{0,\Gamma})\in L^2_{\mathcal{F}_0}(\Omega;\mathbb{L}^2)$, the associated solution of \eqref{forq5.1} satisfies  that
		\begin{align}\label{carfro5.2}
			\begin{aligned}
				&\;\lambda^3\mathbb{E}\int_Q \theta^2\gamma^3q^2\,dx\,dt +\lambda^3\mathbb{E}\int_\Sigma \theta^2\gamma^3q_\Gamma^2\,d\sigma\,dt\\
				&+\lambda\mathbb{E}\int_Q \theta^2\gamma \vert\nabla q\vert^2\,dx\,dt+ \lambda\mathbb{E}\int_\Sigma \theta^2\gamma \vert\nabla_\Gamma q_\Gamma\vert^2 \,d\sigma\,dt
				\\&\leq C \Bigg[ \lambda^3\mathbb{E}\int_0^T\int_{G_1} \theta^2\gamma^3 q^2 \,dx\,dt+ \mathbb{E}\int_Q \theta^2f^2 \,dx\,dt+\mathbb{E}\int_\Sigma \theta^2f_\Gamma^2 \,d\sigma\,dt\\
				&\hspace{0.8cm}+\lambda^2\mathbb{E}\int_Q \theta^2\gamma^2g^2 \,dx\,dt+\lambda^2\mathbb{E}\int_\Sigma \theta^2\gamma^2g_\Gamma^2 \,d\sigma\,dt\Bigg]. 
		\end{aligned}\end{align}
	\end{lm}
	
	On the other hand, we also consider the following backward stochastic parabolic equation
	\begin{equation}\label{eqqgbc}
		\begin{cases}
			dp+\Delta p\,dt=F \,dt+ Z \,dW(t) & \text { in }Q, \\ 
			d p_{\Gamma}+\Delta_\Gamma p_\Gamma \,dt-\partial_\nu p \,dt=F_{\Gamma} \,dt+ \widehat{Z}\,dW(t) & \text { on }\Sigma, \\ 
			p_{\Gamma}(t, x)=p|_\Gamma(t, x) & \text { on }\Sigma, \\
			(p, p_\Gamma)|_{t=T}=(p_T, p_{T,\Gamma}) & \text { in } G\times\Gamma,
		\end{cases}
	\end{equation}
	where $(p_T,p_{T,\Gamma})\in L^2_{\mathcal{F}_T}(\Omega;\mathbb{L}^2)$ is the terminal state, $F\in L^2_\mathcal{F}(0,T;L^2(G))$ and $F_\Gamma\in L^2_\mathcal{F}(0,T;L^2(\Gamma))$. From \cite[Theorem 4.1]{elgrou22SPEwithDBC}, we also recall the following Carleman estimate for equation \eqref{eqqgbc}.
	\begin{lm}\label{lmm4.33}
		There exist a large $\mu_2\geq1$ such that for $\mu=\mu_2$, one can find a constant $C>0$ and a large $\lambda_2\geq1$ depending only on $G$, $G_1$, $\mu_2$ and $T$ such that for all $\lambda\geq\lambda_2$, $F\in L^2_\mathcal{F}(0,T;L^2(G))$, $F_\Gamma\in L^2_\mathcal{F}(0,T;L^2(\Gamma))$, and $(p_T,p_{T,\Gamma})\in L^2_{\mathcal{F}_T}(\Omega;\mathbb{L}^2)$, the corresponding solution $(p,p_\Gamma, P, \widehat{P})$ of \eqref{eqqgbc} satisfies that
		\begin{align}\label{carback5.8}
			\begin{aligned}
				&\;\lambda^3\mathbb{E}\int_Q \theta^2\gamma^3p^2\,dx\,dt +\lambda^3\mathbb{E}\int_\Sigma \theta^2\gamma^3p_\Gamma^2\,d\sigma\,dt\\
				&+\lambda\mathbb{E}\int_Q \theta^2\gamma \vert\nabla p\vert^2\,dx\,dt+ \lambda\mathbb{E}\int_\Sigma \theta^2\gamma \vert\nabla_\Gamma p_\Gamma\vert^2 \,d\sigma\,dt\\
				&\leq C \Bigg[ \lambda^3\mathbb{E}\int_0^T\int_{G_1} \theta^2\gamma^3 p^2 \,dx\,dt+ \mathbb{E}\int_Q \theta^2 F^2 \,dx\,dt+\mathbb{E}\int_\Sigma \theta^2 F_\Gamma^2\, \,d\sigma\,dt\\
				&\qquad\;\;+\lambda^2\mathbb{E}\int_Q \theta^2\gamma^2 P^2 \,dx\,dt+\lambda^2\mathbb{E}\int_\Sigma \theta^2\gamma \widehat{P}^2 \,\,d\sigma\,dt\Bigg]. 
			\end{aligned}
		\end{align}
	\end{lm}
	
	In the sequel, we fix $\mu = \mu^* = \max(\mu_1, \mu_2)$, where $\mu_1$ (resp., $\mu_2$) is the constant given in Lemma~\ref{lmm4.2} (resp., Lemma~\ref{lmm4.33}), and we choose the set $G_1$ given in Lemma~\ref{lmm5.1} such that $G_1 \Subset G_0 \cap \mathcal{O}$. We have the following new Carleman estimate for the coupled system \eqref{adback4.77}-\eqref{adjoforr4.8}.
	\begin{thm}\label{carlfor insencontro}
		There exists a constant $C>0$ and a large $\lambda^*\geq1$ depending only on $G$, $G_0$, $G_1$, $\mu^*$ and $T$ such that for all $\lambda\geq\lambda^*$, and $(q_0,q_{0,\Gamma})\in L^2_{\mathcal{F}_0}(\Omega;\mathbb{L}^2)$, the associated solution $(p,p_\Gamma,P,\widehat{P};q,q_\Gamma)$ of the system \eqref{adback4.77}-\eqref{adjoforr4.8} satisfies that
		\begin{align}\label{carlestcasca}
			\begin{aligned}
				&\,\lambda^2\mathbb{E}\int_Q \theta^2\gamma^3p^2\,dx\,dt + \lambda^2\mathbb{E}\int_\Sigma \theta^2\gamma^3p_\Gamma^2\,d\sigma\,dt+\lambda^2\mathbb{E}\int_Q \theta^2\gamma^3q^2\,dx\,dt \\
				&+\lambda^2\mathbb{E}\int_\Sigma \theta^2\gamma^3q_\Gamma^2\,d\sigma\,dt+\mathbb{E}\int_Q \theta^2\gamma \vert\nabla p\vert^2\,dx\,dt + \mathbb{E}\int_\Sigma \theta^2\gamma \vert\nabla_\Gamma p_\Gamma\vert^2 \,d\sigma\,dt\\
				&+\mathbb{E}\int_Q \theta^2\gamma \vert\nabla q\vert^2\,dx\,dt+ \mathbb{E}\int_\Sigma \theta^2\gamma \vert\nabla_\Gamma q_\Gamma\vert^2 \,d\sigma\,dt\\
				&\leq C \Bigg[ \lambda^{12}\mathbb{E}\int_{Q_0} \theta^2\gamma^{7} p^2 \,dx\,dt +\lambda\mathbb{E}\int_Q \theta^2\gamma^2 P^2 \,dx\,dt + \lambda\mathbb{E}\int_\Sigma \theta^2\gamma \widehat{P}\,^2 \,d\sigma\,dt\Bigg].
			\end{aligned}
		\end{align}
	\end{thm}
	\begin{proof}
		We first apply the Carleman estimate \eqref{carback5.8} for equation \eqref{adback4.77}, then there exists a constant $C>0$ and a large $\lambda_2\geq1$ such that for all $\lambda\geq\lambda_2$, we have 
		
		\begin{align}\label{estt4.3f}
			\begin{aligned}
				&\,\lambda^3\mathbb{E}\int_Q \theta^2\gamma^3p^2\,dx\,dt + \lambda^3\mathbb{E}\int_\Sigma \theta^2\gamma^3p_\Gamma^2\,d\sigma\,dt\\
				&+\lambda\mathbb{E}\int_Q \theta^2\gamma \vert\nabla p\vert^2\,dx\,dt + \lambda\mathbb{E}\int_\Sigma \theta^2\gamma \vert\nabla_\Gamma p_\Gamma\vert^2 \,d\sigma\,dt\\
				&\leq C \Bigg[ \lambda^3\mathbb{E}\int_0^T\int_{G_1} \theta^2\gamma^3 p^2 \,dx\,dt + \mathbb{E}\int_0^T\int_\mathcal{O} \theta^2 q^2 \,dx\,dt \\
				&
				\qquad\;+\mathbb{E}\int_0^T\int_{\mathcal{O}_\Gamma} \theta^2 q_\Gamma^2 \,d\sigma\,dt+\lambda^2\mathbb{E}\int_Q \theta^2\gamma^2 P^2 \,dx\,dt + \lambda^2\mathbb{E}\int_\Sigma \theta^2\gamma \widehat{P}\,^2 \,d\sigma\,dt\Bigg].
			\end{aligned}
		\end{align}
		
		On the other hand, by applying the Carleman estimate \eqref{carfro5.2} for equation \eqref{adjoforr4.8}, there exists a  constant $C>0$ and a large $\lambda_1\geq1$ so that for all $\lambda\geq\lambda_1$
		\begin{align}\label{estimm5.4}
			\begin{aligned}
				&\;\lambda^3\mathbb{E}\int_Q \theta^2\gamma^3q^2\,dx\,dt +\lambda^3\mathbb{E}\int_\Sigma \theta^2\gamma^3q_\Gamma^2\,d\sigma\,dt\\
				&+\lambda\mathbb{E}\int_Q \theta^2\gamma \vert\nabla q\vert^2\,dx\,dt+ \lambda\mathbb{E}\int_\Sigma \theta^2\gamma \vert\nabla_\Gamma q_\Gamma\vert^2 \,d\sigma\,dt
				\\&\leq C \lambda^3\mathbb{E}\int_0^T\int_{G_1} \theta^2\gamma^3 q^2 \,dx\,dt. 
		\end{aligned}\end{align}
		
		In the rest of the proof, we take $\lambda \geq \lambda^* = \max(\lambda_1, \lambda_2)$, which, if necessary, will be chosen to be large enough to absorb some lower-order terms. Let us now estimate the term on the right hand side of \eqref{estimm5.4}. For that, let us choose $\zeta\in C^{\infty}(\mathbb{R}^N)$ such that
		\begin{subequations}\label{assmzeta}
			\begin{align}
				&0\leq\zeta\leq 1,\,\quad \zeta=1 \,\, \textnormal{in}\,\, G_1,\quad \,\textnormal{supp}(\zeta)\subset G_0,\label{assmzeta1}\\ &\quad
				\frac{\Delta\zeta}{\zeta^{1/2}}\in L^\infty(G),\quad
				\frac{\nabla\zeta}{\zeta^{1/2}}\in L^\infty(G;\mathbb{R}^N)\label{assmzeta2}.
			\end{align}
		\end{subequations}
		Such a function $\zeta$ exists. Indeed, by standard arguments, one can take a cut-off function $\zeta_0 \in C^\infty_0(\mathbb{R}^N)$ satisfying \eqref{assmzeta1} and then choosing $\zeta = \zeta_0^4$, which easily verifies \eqref{assmzeta}. By using Itô's formula for $d(\zeta\theta^2\gamma^3pq)$, we have
		
		\begin{align}\label{ienbyItoformula}
			d(\zeta\theta^2\gamma^3pq)=\zeta(\theta^2\gamma^3)_tpq dt+\zeta\theta^2\gamma^3\left[(-\Delta p-\chi_{\mathcal{O}}q)q+p\Delta q \right]dt+\{...\}dW(t).
		\end{align}
		Integrating \eqref{ienbyItoformula} on $Q$ and taking the expectation on both sides, we deduce that
		\begin{align*}
			\mathbb{E}\int_Q \zeta\theta^2\gamma^3\chi_{\mathcal{O}}q^2 \,dx\,dt= \mathbb{E}\int_Q \zeta(\theta^2\gamma^3)_t pq \,dx\,dt+\mathbb{E}\int_Q \zeta\theta^2\gamma^3(-q\Delta p+p\Delta q) \,dx\,dt,
		\end{align*}
		which implies that
		\begin{align*}
			\begin{aligned}
				\mathbb{E}\int_0^T\int_{G_1} \theta^2\gamma^3q^2 \,dx\,dt&\leq \mathbb{E}\int_Q \zeta(\theta^2\gamma^3)_t pq \,dx\,dt+\mathbb{E}\int_Q \zeta\theta^2\gamma^3(-q\Delta p+p\Delta q) \,dx\,dt\\
				&=I+J.
			\end{aligned}
		\end{align*}
		Fix a parameter $\rho>0$. Taking a large enough $\lambda$, it is easy to see that
		$$|(\theta^2\gamma^3)_t|\leq C\lambda\theta^2\gamma^5.$$
		By Young's inequality, it follows that
		\begin{align}\label{I1165}
			I\leq \rho\mathbb{E}\int_Q \theta^2\gamma^3 q^2 \,dx\,dt+C(\rho)\lambda^{2}\mathbb{E}\int_{Q_0} \theta^2\gamma^7p^2 \,dx\,dt.
		\end{align}
		On the other hand, by integration by parts, it is not difficult to verify that
		\begin{align}\label{functioJ}
			J=-\lambda^3\mathbb{E}\int_Q \gamma^3\Delta(\zeta\theta^2)pq \,dx\,dt-2\lambda^3\mathbb{E}\int_Q \gamma^3 p\nabla q\cdot\nabla(\zeta\theta^2) \,dx\,dt.
		\end{align}
		Recalling \eqref{assmzeta} and taking large $\lambda$, we have that
		\begin{align}\label{assumpongradLapl}
			|\nabla(\zeta\theta^2)|\leq C\lambda\zeta^{1/2}\theta^2\gamma,\qquad|\Delta(\zeta\theta^2)|\leq C\lambda^2\zeta^{1/2}\theta^2\gamma^2.
		\end{align}
		Using \eqref{assumpongradLapl} in \eqref{functioJ}, it follows that
		\begin{align*}
			J\leq C\lambda^5\mathbb{E}\int_Q \zeta^{1/2}\theta^2\gamma^5|p||q|\,dx\,dt+C\lambda^4\mathbb{E}\int_Q \zeta^{1/2}\theta^2\gamma^4|p||\nabla q|\,dx\,dt,
		\end{align*}
		which provides that
		\begin{align}\label{I2258}
			J\leq\rho\mathbb{E}\int_Q \theta^2\gamma^3 q^2 \,dx\,dt+\rho\lambda^{-2}\mathbb{E}\int_Q \theta^2\gamma |\nabla q|^2 \,dx\,dt+C(\rho)\lambda^{10}\mathbb{E}\int_{Q_0} \theta^2\gamma^{7} p^2 \,dx\,dt.
		\end{align}
		From \eqref{I1165} and \eqref{I2258}, we get 
		\begin{align}\label{4.111ineq}
			\begin{aligned}
				\lambda^3\mathbb{E}\int_0^T\int_{G_1} \theta^2\gamma^3q^2 \,dx\,dt&\leq 
				2\rho\lambda^3\mathbb{E}\int_Q \theta^2\gamma^3 q^2 \,dx\,dt+C(\rho)\lambda^{5}\mathbb{E}\int_{Q_0} \theta^2\gamma^7p^2 \,dx\,dt\\
				&\quad+\rho\lambda\mathbb{E}\int_Q \theta^2\gamma |\nabla q|^2 \,dx\,dt+C(\rho)\lambda^{13}\mathbb{E}\int_{Q_0} \theta^2\gamma^{7} p^2 \,dx\,dt.
			\end{aligned}
		\end{align}
		Taking a large $\lambda$ in \eqref{4.111ineq}, we obtain that
		\begin{align}\label{est511}
			\begin{aligned}
				\lambda^3\mathbb{E}\int_0^T\int_{G_1} \theta^2\gamma^3q^2 \,dx\,dt&\leq 
				2\rho\lambda^3\mathbb{E}\int_Q \theta^2\gamma^3 q^2 \,dx\,dt+\rho\lambda\mathbb{E}\int_Q \theta^2\gamma |\nabla q|^2 \,dx\,dt\\
				&\quad+C(\rho)\lambda^{13}\mathbb{E}\int_{Q_0} \theta^2\gamma^{7} p^2 \,dx\,dt.
			\end{aligned}
		\end{align}
		Combining \eqref{est511} and \eqref{estimm5.4}, and taking a sufficiently small $\rho$, we end up with
		\begin{align}\label{finestt512}
			\begin{aligned}
				&\;\lambda^3\mathbb{E}\int_Q \theta^2\gamma^3q^2\,dx\,dt +\lambda^3\mathbb{E}\int_\Sigma \theta^2\gamma^3q_\Gamma^2\,d\sigma\,dt\\
				&+\lambda\mathbb{E}\int_Q \theta^2\gamma \vert\nabla q\vert^2\,dx\,dt+ \lambda\mathbb{E}\int_\Sigma \theta^2\gamma \vert\nabla_\Gamma q_\Gamma\vert^2 \,d\sigma\,dt
				\\&\leq C  \lambda^{13}\mathbb{E}\int_{Q_0} \theta^2\gamma^{7} p^2 \,dx\,dt. 
		\end{aligned}\end{align}
		Adding \eqref{finestt512} and \eqref{estt4.3f}, and taking a large enough $\lambda$, we finally deduce our desired Carleman estimate \eqref{carlestcasca}. This concludes the proof of Theorem \ref{carlfor insencontro}.
	\end{proof}
	\subsection{Observability inequality}
	Based on the Carleman estimate \eqref{carlestcasca}, we  have the following observability inequality.
	\begin{prop}\label{propo5.1obseine}
		There exists positive constants $C$ and $M$ depending only on $G, G_0, \mathcal{O}, T, a_1, a_2, b_1$ and $b_2$ such that for any $(q_0,q_{0,\Gamma})\in L^2_{\mathcal{F}_0}(\Omega;\mathbb{L}^2)$, the solution $(p,p_\Gamma,P,\widehat{P};q,q_\Gamma)$ of the coupled system \eqref{adback4.77}-\eqref{adjoforr4.8} satisfies that
		\begin{align}\label{obseine4.9}
			\begin{aligned}
				&\mathbb{E}\int_Q \textnormal{exp}\left(-Mt^{-1}\right) p^2 \,dx\,dt+\mathbb{E}\int_\Sigma \textnormal{exp}\left(-Mt^{-1}\right) p_\Gamma^2 \,d\sigma\,dt\\
				&\leq C\left(\mathbb{E}\int_{Q_0} p^2 \,dx\,dt+\mathbb{E}\int_{Q} P^2 \,dx\,dt+\mathbb{E}\int_{\Sigma} \widehat{P}^2 \,d\sigma\,dt\right).
			\end{aligned}
		\end{align}
	\end{prop}
	\begin{proof}
		In what follows, we fix $\lambda=\lambda^*$ given in Theorem \ref{carlfor insencontro}, and we first derive the usual energy estimate to the forward equation \eqref{adjoforr4.8}. Let $t_1,t_2\in(0,T)$ such that $t_1<t_2$, then  using Itô's formula, we compute $d|(q,q_\Gamma)|_{\mathbb{L}^2}^2$, integrating the obtained equality on $(t_1,t_2)$ and taking the expectation on both sides, we get
		\begin{align*}
			\mathbb{E}\int_{t_1}^{t_2}\int_G dq^2 \,dx +   \mathbb{E}\int_{t_1}^{t_2}\int_\Gamma dq_\Gamma^2 \,d\sigma&=-2\mathbb{E}\int_{t_1}^{t_2}\int_G |\nabla q|^2 \,dx\,dt+2\mathbb{E}\int_{t_1}^{t_2}\int_G a_1q^2 \,dx\,dt\\
			&\hspace{0.4cm}+\mathbb{E}\int_{t_1}^{t_2}\int_G a_2^2q^2 \,dx\,dt-2\mathbb{E}\int_{t_1}^{t_2}\int_\Gamma |\nabla_\Gamma q_\Gamma|^2\,d\sigma\,dt\\
			&\hspace{0.4cm}+2\mathbb{E}\int_{t_1}^{t_2}\int_\Gamma b_1 q_\Gamma^2 \,d\sigma\,dt+\mathbb{E}\int_{t_1}^{t_2}\int_\Gamma b_2^2 q_\Gamma^2 \,d\sigma\,dt,
		\end{align*}
		which implies that
		\begin{align}\label{ineq4.155}
			\begin{aligned}
				\mathbb{E}\int_G q^2(t_2)\,dx+\mathbb{E}\int_\Gamma q_\Gamma^2(t_2)\,d\sigma&\leq \mathbb{E}\int_G q^2(t_1)dx+\mathbb{E}\int_\Gamma q_\Gamma^2(t_1)d\sigma\\
				&\quad+C\left(\mathbb{E}\int_{t_1}^{t_2}\int_G q^2 \,dx\,dt+\mathbb{E}\int_{t_1}^{t_2}\int_\Gamma q_\Gamma^2\,d\sigma\,dt\right).
			\end{aligned}
		\end{align}
		Using Gronwall's inequality for \eqref{ineq4.155}, we obtain that
		$$\mathbb{E}\int_G q^2(t_2)\,dx+\mathbb{E}\int_\Gamma q_\Gamma^2(t_2)\,d\sigma\leq C\left(\mathbb{E}\int_G q^2(t_1)\,dx+\mathbb{E}\int_\Gamma q_\Gamma^2(t_1)\,d\sigma\right).$$
		Then it follows that for any $t\in(T/4,3T/4)$
		\begin{align}\label{inee4.1666i}
			\mathbb{E}\int_G q^2\left(t+\frac{T}{4}\right)\,dx+\mathbb{E}\int_\Gamma q_\Gamma^2\left(t+\frac{T}{4}\right)\,d\sigma\leq C\left(\mathbb{E}\int_G q^2(t)\,dx+\mathbb{E}\int_\Gamma q_\Gamma^2(t)\,d\sigma\right).
		\end{align}
		Integrating \eqref{inee4.1666i} with respect to  $t\in(T/4,3T/4)$, we find that
		\begin{align}\label{estimmqp11}
			\mathbb{E}\int_{T/2}^T\int_G q^2 \,dx\,dt+\mathbb{E}\int_{T/2}^T\int_\Gamma q_\Gamma^2\,d\sigma\,dt\leq C\left(\mathbb{E}\int_{T/4}^{3T/4}\int_G q^2 \,dx\,dt+\mathbb{E}\int_{T/4}^{3T/4}\int_\Gamma q_\Gamma^2 \,d\sigma\,dt\right).
		\end{align}
		
		On the other hand, for any $t\in(0,T)$, we compute $d|(p,p_\Gamma)|_{\mathbb{L}^2}^2$, integrating the obtained equality on $(t,T)$ and taking the expectation on both sides, we arrive at
		\begin{align*}
			\mathbb{E}\int_t^T\int_G dp^2 \,dx+\mathbb{E}\int_t^T\int_\Gamma dp_\Gamma^2 \,d\sigma&=2\mathbb{E}\int_t^T\int_G |\nabla p|^2 \,dx\,ds-2\mathbb{E}\int_t^T\int_G a_1p^2 \,dx\,ds\\
			&\hspace{0.4cm}-2\mathbb{E}\int_t^T\int_G a_2pP \,dx\,ds-2\mathbb{E}\int_t^T\int_G \chi_\mathcal{O} pq \,dx\,ds\\
			&\hspace{0.4cm}+\mathbb{E}\int_t^T\int_G P^2 \,dx\,ds+2\mathbb{E}\int_t^T\int_G |\nabla_\Gamma p_\Gamma|^2 \,d\sigma \,ds\\
			&\hspace{0.4cm}-2\mathbb{E}\int_t^T\int_\Gamma b_1 p_\Gamma^2 \,d\sigma \,ds -2\mathbb{E}\int_t^T\int_\Gamma b_2 p_\Gamma \widehat{P} \,d\sigma \,ds\\
			&\hspace{0.4cm}-2\mathbb{E}\int_t^T\int_\Gamma \chi_{\mathcal{O}_\Gamma} p_\Gamma q_\Gamma \,d\sigma \,ds+\mathbb{E}\int_t^T\int_\Gamma |\widehat{P}|^2 \,d\sigma \,ds,
		\end{align*}
		which gives
		\begin{align}\label{ineqq1.18}
			\begin{aligned}
				-\mathbb{E}\int_t^T\int_G dp^2 \,dx-\mathbb{E}\int_t^T\int_\Gamma dp_\Gamma^2 \,d\sigma\leq &C\Bigg(\mathbb{E}\int_t^T\int_G p^2 \,dx \,ds + \mathbb{E}\int_t^T\int_\Gamma p_\Gamma^2 \,d\sigma \,ds \\
				&\hspace{0.5cm}+ \mathbb{E}\int_t^T\int_G q^2 \,dx \,ds + \mathbb{E}\int_t^T\int_\Gamma q_\Gamma^2 \,d\sigma \,ds\Bigg).
			\end{aligned}
		\end{align}
		Since $(p(T,\cdot),p_\Gamma(T,\cdot))=(0,0)$, then applying Gronwall's inequality in \eqref{ineqq1.18}, we deduce that
		$$\mathbb{E}\int_G p^2(t)\,dx+\mathbb{E}\int_\Gamma p_\Gamma^2(t) \,d\sigma\leq C\left(\mathbb{E}\int_t^T\int_G q^2 \,dx \,ds+\mathbb{E}\int_t^T\int_\Gamma q_\Gamma^2 \,d\sigma \,ds\right).$$
		It follows that  for any $t\in(T/2,T)$
		\begin{align}\label{estimmqp22}
			\mathbb{E}\int_{T/2}^T\int_G p^2 \,dx\,dt+\mathbb{E}\int_{T/2}^T\int_\Gamma p_\Gamma^2\,d\sigma\,dt\leq C\left(\mathbb{E}\int_{T/2}^{T}\int_G q^2 \,dx\,dt+\mathbb{E}\int_{T/2}^{T}\int_\Gamma q_\Gamma^2 \,d\sigma\,dt\right).
		\end{align}
		It is not difficult to see that there exists a large constant $M$ so that
		$$\textnormal{exp}(-Mt^{-1})\leq C \textnormal{exp}\left(2\lambda^*\frac{e^{\mu^*\psi(x)}-e^{2\mu^*|\psi|_\infty}}{t(T-t)}\right)\frac{1}{t^3(T-t)^3},\quad\textnormal{in}\quad(0,T/2)\times\overline{G},$$
		which implies that
		\begin{align}\label{estimmqpfir11}
			\begin{aligned}
				&\mathbb{E}\int_0^{T/2}\int_G \textnormal{exp}(-Mt^{-1})p^2 \,dx\,dt+\mathbb{E}\int_0^{T/2}\int_\Gamma \textnormal{exp}(-Mt^{-1})p_\Gamma^2\,d\sigma\,dt\\
				&\leq C\left(\mathbb{E}\int_Q \theta^2\gamma^3 p^2 \,dx\,dt+\mathbb{E}\int_\Sigma \theta^2\gamma^3 p_\Gamma^2 \,d\sigma\,dt\right).
			\end{aligned}
		\end{align}
		Using our Carleman estimate \eqref{carlestcasca} in the right hand side of \eqref{estimmqpfir11}, we obtain that
		\begin{align}\label{firsesto518}
			\begin{aligned}
				&\mathbb{E}\int_0^{T/2}\int_G \textnormal{exp}(-Mt^{-1})p^2 \,dx\,dt+\mathbb{E}\int_0^{T/2}\int_\Gamma \textnormal{exp}(-Mt^{-1})p_\Gamma^2\,d\sigma\,dt\\
				&\leq C \left( \mathbb{E}\int_{Q_0} p^2 \,dx\,dt +\mathbb{E}\int_Q P^2 \,dx\,dt + \mathbb{E}\int_\Sigma \widehat{P}\,^2 \,d\sigma\,dt\right).
			\end{aligned}
		\end{align}
		On the other hand, combining \eqref{estimmqp11} and \eqref{estimmqp22}, we have 
		\begin{align}\label{estimmqp}
			\mathbb{E}\int_{T/2}^T\int_G p^2 \,dx\,dt+\mathbb{E}\int_{T/2}^T\int_\Gamma p_\Gamma^2\,d\sigma\,dt\leq C\left(\mathbb{E}\int_{T/4}^{3T/4}\int_G q^2 \,dx\,dt+\mathbb{E}\int_{T/4}^{3T/4}\int_\Gamma q_\Gamma^2 \,d\sigma\,dt\right).
		\end{align}
		Notice that
		\begin{align*}
			\begin{aligned}
				&\mathbb{E}\int_{T/2}^T\int_G \textnormal{exp}(-Mt^{-1})p^2 \,dx\,dt+\mathbb{E}\int_{T/2}^T\int_\Gamma \textnormal{exp}(-Mt^{-1})p_\Gamma^2\,d\sigma\,dt\\
				&\leq \mathbb{E}\int_{T/2}^T\int_G p^2 \,dx\,dt+\mathbb{E}\int_{T/2}^T\int_\Gamma p_\Gamma^2\,d\sigma\,dt,
			\end{aligned}
		\end{align*}
		which, together with \eqref{estimmqp}, yields
		\begin{align}\label{ineqq4.244i}
			\begin{aligned}
				&\mathbb{E}\int_{T/2}^T\int_G \textnormal{exp}(-Mt^{-1})p^2 \,dx\,dt+\mathbb{E}\int_{T/2}^T\int_\Gamma \textnormal{exp}(-Mt^{-1})p_\Gamma^2\,d\sigma\,dt\\
				&\leq C\left(\mathbb{E}\int_{T/4}^{3T/4}\int_G \theta^2\gamma^3 q^2 \,dx\,dt+\mathbb{E}\int_{T/4}^{3T/4}\int_\Gamma \theta^2\gamma^3 q_\Gamma^2 \,d\sigma\,dt\right).
			\end{aligned}
		\end{align}
		Using again the Carleman estimate \eqref{carlestcasca} in \eqref{ineqq4.244i}, we conduce that
		\begin{align}\label{secestim522}
			\begin{aligned}
				&\mathbb{E}\int_{T/2}^T\int_G \textnormal{exp}(-Mt^{-1})p^2 \,dx\,dt+\mathbb{E}\int_{T/2}^T\int_\Gamma \textnormal{exp}(-Mt^{-1})p_\Gamma^2\,d\sigma\,dt\\
				&\leq C \left( \mathbb{E}\int_{Q_0} p^2 \,dx\,dt +\mathbb{E}\int_Q P^2 \,dx\,dt + \mathbb{E}\int_\Sigma \widehat{P}\,^2 \,d\sigma\,dt\right).
			\end{aligned}
		\end{align}
		Adding  \eqref{secestim522} and \eqref{firsesto518}, we finally deduce the desired observability inequality \eqref{obseine4.9}. This completes the proof of Proposition \ref{propo5.1obseine}.
	\end{proof}
	\section{Main result: Proof of Theorem \ref{thmm1.3ins}}\label{sec5}
	This section is devoted to proving our main result on the existence of insensitizing controls for equation \eqref{ass15}, presented in Theorem~\ref{thmm1.3ins}.
	
	\begin{proof}[Proof of Theorem \ref{thmm1.3ins}]
		Let us consider the following linear subspace of the space 
		\( L^2_{\mathcal{F}}(0,T;L^2(G_0)) \times L^2_{\mathcal{F}}(0,T;L^2(G)) \times L^2_{\mathcal{F}}(0,T;L^2(\Gamma)) \):
		\begin{align*}
			\mathcal{H} = \Big\{ &(\chi_{G_0}p, P, \widehat{P}) \mid \; (p, p_\Gamma, P, \widehat{P}; q, q_\Gamma) \; \textnormal{is the solution of} \; \eqref{adback4.77}-\eqref{adjoforr4.8}  
			\;\textnormal{with}\\
			&\quad\textnormal{some} \; (q_0, q_{0,\Gamma}) \in L^2_{\mathcal{F}_0}(\Omega; \mathbb{L}^2) \Big\}
		\end{align*}
		Define the linear functional \(\mathcal{L}\) on \(\mathcal{H}\) as follows:
		\[
		\mathcal{L}(\chi_{G_0}p, P, \widehat{P}) = -\mathbb{E} \int_Q p \xi_1 \, dx \, dt - \mathbb{E} \int_\Sigma p_\Gamma \xi_2 \, d\sigma \, dt.
		\]
		By the observability inequality \eqref{obseine4.9}, the functional $\mathcal{L}$ is bounded on $\mathcal{H}$, and we have  
		\begin{align}\label{esthanban}
			|\mathcal{L}|_{\mathcal{L}(\mathcal{H};\mathbb{R})} \leq C \left( \left| \textnormal{exp}\left(Mt^{-1}\right)\xi_1 \right|_{L^2_\mathcal{F}(0,T;L^2(G))} + \left| \textnormal{exp}\left(Mt^{-1}\right)\xi_2 \right|_{L^2_\mathcal{F}(0,T;L^2(\Gamma))} \right).
		\end{align}  
		Then, using the Hahn–Banach theorem, $\mathcal{L}$ can be extended to a bounded linear functional on the whole space \( L^2_\mathcal{F}(0,T;L^2(G_0)) \times L^2_\mathcal{F}(0,T;L^2(G)) \times L^2_\mathcal{F}(0,T;L^2(\Gamma)) \). For simplicity, we still denote this extension by \(\mathcal{L}\). By \eqref{esthanban} and the Riesz representation theorem, there exist controls \( (u, v_1, v_2) \in L^2_\mathcal{F}(0,T;L^2(G_0)) \times L^2_\mathcal{F}(0,T;L^2(G)) \times L^2_\mathcal{F}(0,T;L^2(\Gamma)) \) such that
		\begin{align}\label{fireq11}
			-\mathbb{E}\int_Q p\xi_1 \,dx\,dt-\mathbb{E}\int_\Sigma p_\Gamma\xi_2 \,d\sigma\,dt=\mathbb{E}\int_{Q_0} up \,dx\,dt+\mathbb{E}\int_{Q} v_1P \,dx\,dt+\mathbb{E}\int_{\Sigma} v_2\widehat{P} \,d\sigma\,dt,
		\end{align}
		and 
		\begin{align}\label{estfocontrll}
			\begin{aligned}
				&|u|_{L^2_\mathcal{F}(0,T;L^2(G_0))}+|v_1|_{L^2_\mathcal{F}(0,T;L^2(G))}+|v_2|_{L^2_\mathcal{F}(0,T;L^2(\Gamma))}\\
				&\leq C\left(\left|\textnormal{exp}\left(Mt^{-1}\right)\xi_1\right|_{L^2_\mathcal{F}(0,T;L^2(G))}+\left|\textnormal{exp}\left(Mt^{-1}\right)\xi_2\right|_{L^2_\mathcal{F}(0,T;L^2(\Gamma))}\right).
			\end{aligned}
		\end{align}
		
		We claim that the above-obtained controls \(u\), \(v_1\), and \(v_2\) are the desired insensitizing controls for the functional \(\Phi\) given in \eqref{functioPhii1.2}. Indeed, by using Itô's formula for equations \eqref{forr4.1}-\eqref{adback4.77}, we compute $d\langle(y,y_\Gamma),(p,p_\Gamma)\rangle_{\mathbb{L}^2}$, integrating the result on $(0,T)$ and taking the expectation on both sides, we obtain that
		\begin{align}\label{eqq1.2}
			\begin{aligned}
				&\,\mathbb{E}\int_Q \xi_1 p \,dx\,dt+\mathbb{E}\int_\Sigma \xi_2 p_\Gamma \,d\sigma\,dt+\mathbb{E}\int_Q \chi_{G_0}up \,dx\,dt+\mathbb{E}\int_Q Pv_1 \,dx\,dt\\
				&+\mathbb{E}\int_\Sigma \widehat{P}v_2 \,d\sigma\,dt-\mathbb{E}\int_Q \chi_{\mathcal{O}}yq \,dx\,dt- \mathbb{E}\int_{\Sigma} \chi_{\mathcal{O}_\Gamma} y_\Gamma q_\Gamma  \,d\sigma\,dt=0.
			\end{aligned}
		\end{align}
		Again by Itô's formula for equations \eqref{back45}-\eqref{adjoforr4.8}, we get
		\begin{align}\label{eqq1.232}
			\begin{aligned}
				-\mathbb{E}\int_G z(0) q_0 \,dx-\mathbb{E}\int_\Gamma z_\Gamma(0) q_{0,\Gamma} \,d\sigma=-\mathbb{E}\int_Q \chi_{\mathcal{O}}yq \,dx\,dt- \mathbb{E}\int_{\Sigma} \chi_{\mathcal{O}_\Gamma} y_\Gamma q_\Gamma  \,d\sigma\,dt.
			\end{aligned}
		\end{align}
		From \eqref{eqq1.2} and \eqref{eqq1.232}, we have that
		\begin{align}\label{eqq6.44}
			\begin{aligned}
				\mathbb{E}\int_G z(0) q_0 \,dx+\mathbb{E}\int_\Gamma z_\Gamma(0) q_{0,\Gamma} \,d\sigma&=\mathbb{E}\int_Q \xi_1 p \,dx\,dt+\mathbb{E}\int_\Sigma \xi_2 p_\Gamma \,d\sigma\,dt+\mathbb{E}\int_{Q_0} up \,dx\,dt\\
				&\quad+\mathbb{E}\int_Q Pv_1 \,dx\,dt+\mathbb{E}\int_\Sigma \widehat{P}v_2 \,d\sigma\,dt.
			\end{aligned}
		\end{align}
		Combining  \eqref{eqq6.44} and \eqref{fireq11}, we conclude that
		$$\mathbb{E}\int_G z(0) q_0 \,dx+\mathbb{E}\int_\Gamma z_\Gamma(0) q_{0,\Gamma} \,d\sigma=0.$$
		Since $(q_0,q_{0,\Gamma})$ can be chosen arbitrarily in $L^2_{\mathcal{F}_0}(\Omega;\mathbb{L}^2)$, we deduce that 
		$$
		z(0,\cdot)=0\;\; \textnormal{in}\;\; G, \quad z_\Gamma(0,\cdot)=0 \;\; \textnormal{on}\;\;  \Gamma, \quad\mathbb{P}\textnormal{-a.s.}
		$$
		Therefore, from \eqref{estfocontrll} and Proposition~\ref{proposs11}, we conclude the proof of Theorem~\ref{thmm1.3ins}.
	\end{proof}

	\section{Conclusion and further comments}
	In this paper, we studied the insensitizing controllability problem \eqref{inspb} for the forward stochastic parabolic equation \eqref{ass15} with dynamic boundary conditions. We first reduced this problem to a classical controllability problem for the coupled system \eqref{forr4.1}-\eqref{back45} of forward-backward stochastic parabolic equations. Equivalently, we demonstrated the observability inequality \eqref{obseine4.9} for the associated adjoint coupled system \eqref{adback4.77}-\eqref{adjoforr4.8}, by proving the new Carleman estimate \eqref{carlestcasca}.
	
	The results presented in this paper can be further developed in many ways. Let us formulate this in the form of some open problems to guide future investigations.
	
	\begin{enumerate}
		\item It would be interesting to study the insensitizing control problem for the following backward stochastic parabolic equation
		\begin{equation}\label{backeqqinse}
			\begin{cases}
				\begin{array}{ll}
					dy + \Delta y \,dt = (\xi_1 + a_1 y + a_2 Y) \,dt + Y \,dW(t), & \textnormal{in } Q, \\
					dy_\Gamma + \Delta_\Gamma y_\Gamma \,dt - \partial_\nu y \,dt = (\xi_2 + b_1 y_\Gamma + b_2 \widehat{Y})\,dt + \widehat{Y} \,dW(t), & \textnormal{on } \Sigma, \\
					y_\Gamma(t,x) = y\vert_\Gamma(t,x), & \textnormal{on } \Sigma, \\
					(y, y_\Gamma)\vert_{t=T} = (y_T + \tau_1 \widehat{y}_T, y_{T,\Gamma} + \tau_2 \widehat{y}_{T,\Gamma}), & \textnormal{in } G \times \Gamma.
				\end{array}
			\end{cases}
		\end{equation}
		
		The problem is to act on the system \eqref{backeqqinse} by one (or more) controls to insensitize the following energy functional
		\begin{align*}
			\Phi(y, y_\Gamma; Y, \widehat{Y}) &= \frac{1}{2}\mathbb{E}\int_0^T\int_{\mathcal{O}_1} |y|^2 \,dx\,dt 
			+ \frac{1}{2}\mathbb{E}\int_0^T\int_{\mathcal{O}_{\Gamma1}} |y_\Gamma|^2 \,d\sigma\,dt \\
			&\quad + \frac{1}{2}\mathbb{E}\int_0^T\int_{\mathcal{O}_2} |Y|^2 \,dx\,dt 
			+ \frac{1}{2}\mathbb{E}\int_0^T\int_{\mathcal{O}_{\Gamma2}} |\widehat{Y}|^2 \,d\sigma\,dt,
		\end{align*}
		where \( (y, y_\Gamma, Y, \widehat{Y})\) is the solution of \eqref{backeqqinse}. The sets \( \mathcal{O}_1 \) and \( \mathcal{O}_2 \) are open subsets of \( G \), while \( \mathcal{O}_{\Gamma1} \) and \( \mathcal{O}_{\Gamma2} \) are open subsets of \( \Gamma \).
		\item An important direction for research would be to extend the insensitizing control problem \eqref{inspb} for equation \eqref{ass15} to the following semilinear stochastic parabolic equation
		\begin{equation*}
			\begin{cases}
				\begin{array}{ll}
					dy - \Delta y \,dt = (\xi_1 + f_1(y) + \chi_{G_0} u) \,dt + (f_2(y) + v_1) \,dW(t), & \textnormal{in } Q, \\
					dy_\Gamma - \Delta_\Gamma y_\Gamma \,dt + \partial_\nu y \,dt = (\xi_2 + g_1(y)) \,dt + (g_2(y) + v_2) \,dW(t), & \textnormal{on } \Sigma, \\
					y_\Gamma(t,x) = y\vert_\Gamma(t,x), & \textnormal{on } \Sigma, \\
					(y, y_\Gamma)\vert_{t=0} = (y_0 + \tau_1 \widehat{y}_0, y_{0,\Gamma} + \tau_2 \widehat{y}_{0,\Gamma}), & \textnormal{in } G \times \Gamma,
				\end{array}
			\end{cases}
		\end{equation*}
		where \( f_i \) and \( g_i \) (\( i=1,2 \)) are globally Lipschitz continuous functions. See \cite{Tereza2000,zhaYinGao} for some results on the deterministic setting.
		\item It is quite interesting to study the following notion of the \( \varepsilon \)-insensitivity control problem: For every \( \varepsilon > 0 \), the control triple \( (u, v_1, v_2) \) is said to \( \varepsilon \)-insensitize the functional \( \Phi \) if
		\begin{align}\label{aespinsenns}
			\left|\frac{\partial \Phi(y, y_\Gamma)}{\partial \tau_1} \Bigg|_{\tau_1 = \tau_2 = 0}\right| \leq \varepsilon 
			\quad \text{and} \quad
			\left|\frac{\partial \Phi(y, y_\Gamma)}{\partial \tau_2} \Bigg|_{\tau_1 = \tau_2 = 0}\right| \leq \varepsilon.
		\end{align}
		
		In Section \ref{sec3}, we have demonstrated that the insensitivity control problem \eqref{inspb} is equivalent to a null controllability problem, while the \( \varepsilon \)-insensitivity control problem \eqref{aespinsenns} can be shown to be equivalent to an approximate controllability property. See \cite{bodafabre95} for related results in the deterministic case.
		\item Another interesting avenue is to address the insensitizing control problem with a more complex functional \( \Phi \) involving the state of \eqref{ass15} and its gradient. For instance, we consider
		\begin{align*}
			\Phi(y, y_\Gamma) &= \frac{1}{2}\mathbb{E}\int_0^T\int_{\mathcal{O}} j(y, \nabla y) \,dx\,dt 
			+ \frac{1}{2}\mathbb{E}\int_0^T\int_{\mathcal{O}_\Gamma} j(y_\Gamma, \nabla_\Gamma y_\Gamma) \,d\sigma\,dt,
		\end{align*}
		where \( j:\mathbb{R}^{1+N}\longrightarrow\mathbb{R} \) is a given function and \( (y, y_\Gamma) \) denotes the solution of \eqref{ass15}. For some works on deterministic parabolic equations, we refer to \cite{gureSiam07, sancarrmora}.
	\end{enumerate}


\begin{thebibliography}{10}
		
		\bibitem{ACMO20}
		E. M. Ait Ben Hassi, S. E. Chorfi, L. Maniar, and O. Oukdach, 
		\newblock Lipschitz stability for an inverse source problem in anisotropic parabolic equations with dynamic boundary conditions,
		\newblock{\em Evolution Equations and Control Theory}, \textbf{10}(2021), 837--859.
		
		\bibitem{elgrou22SPEwithDBC}
		M. Baroun, S. Boulite, A. Elgrou, and L. Maniar,
		\newblock Null controllability for stochastic parabolic equations with dynamic boundary conditions,
		\newblock{\em Journal of Dynamical and Control Systems}, \textbf{29}(2023), 1727--1756.
		
		\bibitem{BackSPEwithDBC}
		M. Baroun, S. Boulite, A. Elgrou, and L. Maniar,
		\newblock Null controllability for backward stochastic parabolic convection-diffusion equations with dynamic boundary conditions, 
		\newblock {\em Mathematical Control and Related Fields}, (2024).
		
		\bibitem{elgrou1D23}
		M. Baroun, S. Boulite, A. Elgrou, and L. Maniar,
		\newblock Null controllability for one-dimensional stochastic heat equations with mixed Dirichlet-dynamic boundary conditions. To appear.
		
		\bibitem{bodafabre95}
		O. Bodart and C. Fabre, 
		\newblock Controls insensitizing the norm of the solution of a semilinear heat-equation, 
		\newblock {\em Journal of Mathematical Analysis and Applications}, \textbf{195}(1995), 658--683.
		
		
		\bibitem{BodarBurgosperez}
		O. Bodart, M. González-Burgos, and R. Pérez-García, 
		\newblock Insensitizing controls for a semilinear heat equation with a superlinear nonlinearity, 
		\newblock {\em Comptes Rendus Mathematique}, \textbf{335}(2002), 677--682.
		
		\bibitem{BodarBurgosPerez2004}
		O. Bodart, M. González-Burgos, and R. Pérez-García,
		\newblock  Existence of insensitizing controls for a semilinear heat equation with a superlinear nonlinearity, 
		\newblock {\em Communications in Partial Differential Equations}, \textbf{29}(2004), 1017--1050.
		
		
		
		
		
		\bibitem{BodarBurgosPerez04NonAnalysis}
		O. Bodart, M. González-Burgos, and R. Pérez-García,
		\newblock Insensitizing controls for a heat equation with a nonlinear term involving the state and the gradient,
		\newblock {\em Nonlinear Analysis: Theory, Methods \& Applications}, \textbf{57}(2004), 687--711.
		
		
		
		
		\bibitem{BodaGnPer}
		O. Bodart, M. González-Burgos, and R. Pérez-García,  
		\newblock Insensitizing controls for a semilinear heat equation with a superlinear nonlinearity, 
		\newblock {\em Comptes Rendus Mathematique}, \textbf{335}(2002), 677--682.
		
		
		\bibitem{withouextra}
		S. Boulite, A. Elgrou, L. Maniar, and O. Oukdach,  
		\newblock Controllability for forward stochastic parabolic equations with dynamic boundary conditions without extra forces,
		\newblock {\em arXiv preprint arXiv:2312.15537}, (2023).
		
		
		\bibitem{BoMaOuCost}
		I. Boutaayamou, S. E. Chorfi, L. Maniar, and O.  Oukdach,  
		\newblock The cost of approximate controllability of heat equation with general dynamical boundary conditions,
		\newblock {\em  Portugaliae Mathematica}, \textbf{78}(2021), 65--99.
		
		\bibitem{BoMaOuNash} 
		I. Boutaayamou, L. Maniar, and O. Oukdach,
		\newblock Stackelberg-Nash null controllability of heat equation with general dynamic boundary conditions, 
		\newblock{\em Evolution Equations and Control Theory}, \textbf{11}(2022), 1285--1307.
		
		
		
		
		
		
		
		
		
		
		
		
		
		
		
		
		
		
		\bibitem{InvStoch}
		A. Elgrou, L. Maniar,  and O. Oukdach,
		\newblock Inverse initial problem under Nash strategy for stochastic reaction-diffusion equations with dynamic boundary conditions,
		\newblock {\em arXiv preprint arXiv:2410.10007,} (2024).
		
		
		
		
		\bibitem{ErveLissPriva}
		S. Ervedoza, P. Lissy, and Y. Privat,
		\newblock Desensitizing control for the heat equation with respect to domain variations,
		\newblock {\em Journal de l’École polytechnique—Mathématiques}, \textbf{9}(2022), 1397--1429.
		
		
		
		
		\bibitem{fursikov1996controllability}
		A. V. Fursikov and O. Yu. Imanuvilov,
		\newblock{ Controllability of evolution equations}, 
		\newblock{\em Lecture Note Series 34, Research Institute of Mathematics, Seoul National University,} (1996).
		
		
		
		
		
		
		\bibitem{gureSiam07}
		S. Guerrero,
		\newblock Null controllability of some systems of two parabolic equations with one control force, 
		\newblock {\em SIAM journal on control and optimization}, \textbf{46}(2007), 379--394.
		
		
		\bibitem{khoutaibi2020null1}
		A. Khoutaibi, L. Maniar, and O.  Oukdach, 
		\newblock Null controllability for semilinear heat equation with dynamic boundary conditions,
		\newblock {\em Discrete  Continuous Dynamical Systems-Series S}, \textbf{15}(2022).
		
		\bibitem{kiran}
		M. Kirane, A. Lopushansky, and  H.  Lopushanska, Determination of two unknown functions of different variables in a time‐fractional differential equation. Mathematical Methods in the Applied Sciences.
		
		
		\bibitem{lions1989quel}
		J.-L. Lions, 
		\newblock Remarques préliminaires sur le contrôle des systemesa données incompletes,
		\newblock {\em In Actas del Congreso de Ecuaciones Diferenciales y Aplicaciones (CEDYA), Universidad de Malaga}, 43--54.
		
		
		
		
		
		\bibitem{liu2014global}
		X. Liu,
		\newblock Global Carleman estimate for stochastic parabolic equations, and its application,
		\newblock{\em ESAIM: Control, Optimisation and Calculus of Variations},
		\textbf{20}(2014), 823--839.
		
		\bibitem{liu2019carleman}
		X. Liu and Y. Yu,
		\newblock Carleman estimates of some stochastic degenerate parabolic equations and application,
		\newblock {\em SIAM Journal on Control and Optimization}, \textbf{57}(2019), 3527--3552.
		
		\bibitem{lu2011some}
		Q. L{\"u},
		\newblock Some results on the controllability of forward stochastic heat equations with control on the drift,
		\newblock {\em Journal of Functional Analysis}, \textbf{260}(2011), 832--851.
		
		\bibitem{lu2021mathematical}
		Q. L{\"u} and X. Zhang,
		\newblock { Mathematical control theory for stochastic partial differential equations}, 
		\newblock{\em Springer,} (2021).
		
		
		\bibitem{maniar2017null}
		L. Maniar, M. Meyries, and R. Schnaubelt,
		\newblock Null controllability for parabolic equations with dynamic boundary conditions,
		\newblock {\em Evolution Equations and Control Theory}, \textbf{6}(2017), 381--407.
		
		
		\bibitem{Spectral}
		L. Maniar, O. Oukdach, and W.  Zouhair,
		\newblock Lebeau–Robbiano inequality for heat equation with dynamic boundary conditions and optimal null controllability,
		\newblock {\em  Differential Equations and Dynamical Systems}, (2023), 1--17.
		
		\bibitem{BoMaOuNash2}
		O. Oukdach, I. Boutaayamou, and L. Maniar, 
		\newblock Hierarchical control problem for the heat equation with dynamic boundary conditions, 
		\newblock {\em IMA Journal of Mathematical Control and Information}, (2024).
		
		\bibitem{BEMOstoch}
		O. Oukdach, S. Boulite, A. Elgrou, and  L. Maniar, 
		\newblock Multi-objective control for stochastic parabolic equations with dynamic boundary conditions,
		\newblock {\em arXiv preprint arXiv:2405.15730}, (2024).
		
		
		
		\bibitem{sancarrmora}
		M. C. Santos, N. Carreño, and R. Morales,  
		\newblock An Insensitizing control problem involving tangential gradient terms for a reaction-diffusion equation with dynamic boundary conditions, 
		\newblock {\em arXiv preprint arXiv:2407.09882}, (2024).
		
		
		\bibitem{tang2009null}
		S. Tang and X. Zhang,
		\newblock Null controllability for forward and backward stochastic parabolic equations,
		\newblock {\em SIAM Journal on Control and Optimization}, \textbf{48}(2009), 2191--2216.
		
		
		\bibitem{Tereza2000}
		L. D. Teresa,  
		\newblock Insensitizing controls for a semilinear heat equation: semilinear heat equation, 
		\newblock {\em Communications in Partial Differential Equations}, \textbf{25}(2000), 39--72.
		
		
		\bibitem{Tereza97Esaim}
		L. D. Teresa, 
		\newblock Controls insensitizing the norm of the solution of a semilinear heat equation in unbounded domains,
		\newblock {\em ESAIM: Control, Optimisation and Calculus of Variations}, \textbf{2}(1997), 125--149.
		
		
		\bibitem{Tere_idenfication}
		L. D. Teresa and E. Zuazua,
		\newblock Identification of the class of initial data for the insensitizing control of the heat equation, 
		\newblock {\em Communication on pure and applied analysis}, \textbf{8}(2009).
		
		
		
		
		
		
		
		\bibitem{yansun2011}
		Y. Yan and F. Sun,
		\newblock Insensitizing controls for a forward stochastic heat equation. 
		\newblock {\em  Journal of mathematical analysis and applications}, \textbf{384}(2011), 138--150.
		
		
		
		
		\bibitem{zhaYinGao}
		M. Zhang, J. Yin, and H. Gao,
		\newblock Insensitizing controls for the parabolic equations with dynamic boundary conditions. 
		\newblock {\em Journal of Mathematical Analysis and Applications}, \textbf{475}(2019), 861--873.
		
		\bibitem{Yamam2009invePrb}
		M. Yamamoto, 
		\newblock Carleman estimates for parabolic equations and applications,
		\newblock {\em Inverse problems}, \textbf{25}(2009), 123013.
		
		\bibitem{angZhong16}
		D. Yang and J. Zhong,
		\newblock Observability inequality of backward stochastic heat equations for measurable sets and its applications,
		\newblock {\em SIAM Journal on Control and Optimization,} \textbf{54}(2016), 1157--1175.
		
		
		
		
		
		
		
		
		
		
	\end{thebibliography}
\end{document}